\documentclass[reqno, 11pt,epsfig,amsfonts]{amsart}
\usepackage{amssymb, amsmath, amsthm}

\usepackage[latin1]{inputenc} 
\usepackage{listings, tcolorbox}

\usepackage{tocvsec2}
\usepackage{color}
\usepackage{epsfig}
\usepackage{amsmath}
\usepackage{amssymb}
\usepackage{amscd}
\usepackage{graphicx}
\usepackage{hyperref}
\usepackage{cleveref}
\usepackage{subfigure}
\usepackage{mathrsfs}
\usepackage[T1]{fontenc}
\usepackage{tikz}
\usepackage{floatrow}
\usepackage{float}
\usepackage{pifont,bm}
\usepackage{indentfirst}
\usepackage{enumerate}
\usepackage{amsmath}

\numberwithin{equation}{section}
\newtheorem{lemma}{Lemma}[section]
\newtheorem{theorem}{Theorem}
\newtheorem{re}{Remark}[section]
\newtheorem{prop}{Proposition}[section]

\newtheorem{cor}{Corollary}[section]

   \newtheorem{coro}{Corollary}[section]

\usepackage{epsfig}

\numberwithin{equation}{section}
\numberwithin{theorem}{section}
\numberwithin{prop}{section}
\numberwithin{lemma}{section}
\numberwithin{re}{section}
\numberwithin{coro}{section}
\theoremstyle{definition}



\thanks{*Corresponding author (sftian@cumt.edu.cn, shoufu2006@126.com).}
\thanks{$^\dagger$ Contributed equally as the first author.}

\begin{document}

\title[Stability of $n$-solitons for the ILW equation]{Stability of $n$-soliton solutions for the Intermediate Long Wave equation}


\author[Lu]{Zhen Lu}

\author[Tian]{Shou-Fu Tian$^{*,\dagger}$}
\address{Zhen Lu, Shou-Fu Tian ( Corresponding author)\newline
School of Mathematics, China University of Mining and Technology, Xuzhou 221116, People's Republic of China}
\email{sftian@cumt.edu.cn, shoufu2006@126.com}

\begin{abstract}
In this work, we focus on the stability of $n$-soliton solutions ($n\in \mathbb{N}, n\geq 1$) to the completely integrable intermediate long wave equation (ILW), which models long internal gravity waves in a stratified fluid of finite depth. We show that the $n$-soliton solutions of the ILW equation form non-isolated constrained minimizers of a variational problem associated with a non-local elliptic equation. To establish this result, we construct a suitable Lyapunov functional and utilize the inverse scattering transform to relate the infinite sequence of conservation laws to the scattering data. Furthermore, we employ the recursion operator derived from the bi-Hamiltonian structure to optimize our analysis. Our analysis demonstrates that the $n$-soliton solutions of the ILW equation are dynamically stable in the space $H^{\frac{n}{2}}(\mathbb{R})$ ($n\in \mathbb{N}, n\geq 1$). Additionally, we establish the orbital stability of double soliton solutions in $H^1(\mathbb{R})$.
\end{abstract}

\maketitle

\section{Introduction}\label{sec.1}
In this work, we consider the stability of $n$-soliton solutions for the intermediate long wave (ILW) equation\cite{J,KAS,KKD}
\begin{align}
	u_t+\frac{1}{\delta}u_x+2uu_x+T^\delta u_{xx}=0,\quad  (t,x)\in \mathbb{R} \times \mathbb{R},\label{ILW}
\end{align}
where $u=u(t,x)\in\mathbb{R}$ is a real-valued function, $T^\delta $ is the inverse of Tilbert  transform $T$ with a large scale parameter $\delta>0$. These two operators are defined as
\begin{align}
	(T^\delta f)(x)&=\frac{1}{2\delta}\text{P.V.}\int_{\mathbb{R}}\coth \frac{\pi (y-x)}{2\delta}f(y)dy,\\\nonumber
		(Tf)(x)&=\frac{1}{2\delta}\text{P.V.}\int_{\mathbb{R}}\mathrm{cosech} \frac{\pi (y-x)}{2\delta}f(y)dy,
\end{align}
and P.V. indicates that the integral is to be computed in the principle value sense. Moreover, $T^\delta$ is a zero Fourier multiplier, in the sense that $\partial_x T^\delta$ is the multiplier with symbol
\begin{align}
	\sigma(\partial_x T^\delta)=\widehat{\partial_x T^\delta}=-2\pi\xi\coth(2\pi\delta \xi).
\end{align}

The ILW equation (\ref{ILW}) describes weakly nonlinear internal wave propagation in stratified fluids of finite depth $\delta$. First derived by Kubota, Ko and Dobbs \cite{KKD}, this model captures essential nonlinear and dispersive characteristics of wave dynamics. Joseph \cite{J} subsequently formalized the derivation of the ILW equation by integrating the linear dispersion relation from \cite{P} within Whitham's nonlocal framework \cite{W}, and explicitly constructed its solitary wave solution. This represents a key advancement in establishing the integrability of the equation. Furthermore, the ILW equation itself arises in the context of the two-layer internal wave system \cite{BLS, CGK}. A rigorous derivation of its system form in one or two spatial dimensions, including the case with a free upper surface, can be found in \cite{Xu}.

In addition, the ILW equation (\ref{ILW}) establishes a fundamental connection between two canonical models in water wave theory: the Korteweg-de Vries (KdV) equation for shallow water and the Benjamin-Ono (BO) equation for deep water \cite{ABFS, AFS,HKO,SAF}, with shared initial data in appropriate Sobolev spaces. Specifically, the asymptotic behavior of the ILW solution is governed by the depth parameter $\delta$. Let $u(t,x)$ be the solution of the ILW equation (\ref{ILW}). Then, under the scaling transformation
\begin{align}
	v = \frac{3}{\delta} u\left( \frac{3}{\delta}t,x\right), \label{T1}
\end{align}
the rescaled function $v$ converges (in a suitable functional sense, typically distributional or weak convergence) to the solution of the KdV equation
\begin{align}
	u_t + 2uu_x + \frac{\delta}{3}u_{xxx} = 0, \label{KdV}
\end{align}
in the shallow-water limit ($\delta \to 0$) and in the deep-water limit ($\delta \to \infty$), (\ref{T1}) will converge to the the solution of the BO equation
\begin{align}
	u_t + 2uu_x + Hu_{xx} = 0, \label{BO}
\end{align}
where $H$ denotes the Hilbert transform, defined by the principal value integral
\begin{align*}
	(Hf)(x)=\frac{1}{\pi}\text{P.V.}\int_{\mathbb{R}}\frac{f(y)}{y-x}\mathrm{d}y.
\end{align*}

As a fundamental integrable system \cite{AC,FA,KS}, the ILW equation (\ref{ILW}) possesses an infinite-dimensional completely integrable Hamiltonian structure \cite{KAS, Sa}. This integrability is manifested in the existence of an infinite hierarchy of conservation laws and a Lax-pair formulation.
A key distinction between the ILW equation and the KdV equation lies in their dispersion operators \cite{MPS,Sa,Xu}. Specifically, the ILW equation incorporates a singular nonlocal integro-differential operator $T^\delta$, which fundamentally alters the asymptotic properties of its solutions. In particular, the soliton solutions of the ILW equation exhibit algebraic decay, in contrast to the exponential decay characteristic of KdV solitons. This difference in decay rates represents a key qualitative difference between the two models. Formally, the following quantities are conserved along the flow of the ILW equation \cite{LR}
\begin{align}
	H_0(u)&=:\int_{\mathbb{R}}u\text{d}x, \\
    H_1(u)&=:\frac{1}{2}\int_{\mathbb{R}}u^2\text{d}x, \\
	H_2(u)&=:-\int_{\mathbb{R}}\frac{1}{3}u^3+\frac{1}{2}uT^\delta u_x +\frac{u^2}{2\delta}\text{d}x,\\
	H_3(u)&=:\int_{\mathbb{R}}\frac{1}{4}u^4+\frac{3}{4}u^2T^\delta u_x+\frac{3}{8}(T^\delta u_x)^2+\frac{1}{3\delta}u^3+\frac{1}{2\delta}uT^\delta u_x+\frac{u^2}{8\delta^2}\text{d}x.
\end{align}

It is noteworthy that the energy space for the ILW equation, defined as the domain of the Hamiltonian $H_2(u)$, is $H^{\frac{1}{2}}(\mathbb{R})$.
This functional framework provides the natural setting for investigating solution properties.
In the context of weak solutions, Ginibre and Velo \cite{GV} established the existence of global solutions satisfying
$u \in L^\infty(\mathbb{R}, H^1(\mathbb{R})) \cap L^2_{\mathrm{loc}}(\mathbb{R}, H^{\frac{3}{2}}_{\mathrm{loc}}(\mathbb{R}))$.
For strong solutions, the first global well-posedness result in $H^s(\mathbb{R})$ was obtained by \cite{ABFS} for $s > \frac{3}{2}$.
Subsequent developments progressively lowered the regularity threshold: Molinet and Vento \cite{MV} extended well-posedness to $s > \frac{1}{2}$
(where unconditional uniqueness holds), while Molinet, Pilod, and Vento \cite{MPV} advanced it to $s > \frac{1}{4}$.
A significant breakthrough in the sharp low-regularity theory was achieved by Ifrim and Saut \cite{IS}, who proved global well-posedness
in $L^2(\mathbb{R})$ for small initial data. Most recently, Gassot and Laurens \cite{GAL} established sharp global well-posedness
for the ILW equation in $H^s(\mathbb{T})$ for all $s > -\frac{1}{2}$, thereby
completing the well-posedness theory on the torus. Additionally, they proved the
continuous convergence of solutions to those of the Benjamin-Ono equation in the
deep-water limit $\delta \to \infty$.

Based upon the facts above, the (\ref{ILW}) equation can naturally be viewed as a Hamiltonian system of the form \cite{KSA, LR}
\begin{align}
	u_t=\mathcal{J}\frac{\delta H_2(u)}{\delta u},
\end{align}
where $\mathcal{J}=\partial_x$, and $\frac{\delta H_2(u)}{\delta u}$ (subsequently simplified to $H_2'(u)$) refers to the variational derivative of $H_2(u)$ can be written as
\begin{align*}
	(\frac{\partial}{\partial \epsilon}H_2(u+\epsilon f))\Big|_{\epsilon=0}=\int_{\mathbb{R}}\frac{\delta H_{2}}{\delta u}(x)f(x)\text{d}x.
\end{align*}

Moreover, the ILW equation possesses a bi-Hamiltonian structure \cite{Sa}. However, unlike the KdV equation, the bi-Hamiltonian structure of the ILW equation is highly nontrivial due to its formulation involving both the spatial derivative operator $\partial_x$ and the singular integral operator $T^\delta$. This operator composition aligns the ILW equation with structural features characteristic of completely integrable systems in two spatial dimensions.

To formalize this connection, let subscripts $1, 2$ denote dependence on the independent variables $x:= x_1$ and $x_2$. For arbitrary functions $f_{12}$ and $g_{12}$, we define the associated bilinear form by
\begin{align}
	\langle f_{12},g_{12}\rangle
	&:= \int_{\mathbb{R}^{2}}f_{12}g_{12}^{*} \mathrm{d}x_{1}\mathrm{d}x_{2},\label{1.9}
\end{align}
where the asterisk superscript $^*$ denotes complex conjugation. Define operators in $L^{2}(\mathbb{R}^{2},\mathbb{C})$ (with domain $H^{1}(\mathbb{R}^{2},\mathbb{C})$)
\begin{align}
	\theta_{12}^{\pm}
	&:= u_{1}\pm u_{2} + i(\partial_{x_{1}} \mp \partial_{x_{2}}), \quad u_{j}=u(x_{j},t), \quad j=1,2.\label{1.10}
\end{align}
Then the two compatible Hamiltonian operators for the ILW equation are given by
\begin{align}
	\mathcal{J}_{12}^{(1)}
	:= \theta_{12}^{-}, \quad
	\mathcal{J}_{12}^{(2)}
	:= (i\theta_{12}^{-}T^\delta_{12} - \theta_{12}^{+})\theta_{12}^{-},
\end{align}
where $T^\delta_{12}$ is an extended operator defined as
\begin{align}
	(T^\delta_{12}f_{12})(x_{1},x_{2}) =\frac{1}{2\delta}\mathrm{P.V.}\int_{\mathbb{R}}\coth \left(\frac{\pi}{2\delta}[\xi-(x_1+x_2)]\right)F(\xi,x_1-x_2)\mathrm{d}\xi,\label{1.13}
\end{align}
and
\begin{align}
	f(x_1,x_2)=F(x_1+x_2,x_1-x_2).
\end{align}
Then the ILW hierarchy can be expressed in the form
\begin{align}\label{1.14}
	u_t&=\frac{i}{2^n}\int_{\mathbb{R}}\delta(x_1-x_2)(\mathcal{R}_{12}^\star)^n\theta_{12}^-\cdot1\mathrm{d}x_2\\\nonumber
	&=\frac{i}{2^n}\int_{\mathbb{R}}\delta(x_1-x_2)\theta_{12}^-\mathcal{R}_{12}^n\cdot1\mathrm{d}x_2=\mathcal{J}\frac{\delta H_n(u)}{\delta u},\quad \text{for all} ~n\in \mathbb{N},
\end{align}
where $\star$ denotes the adjoint with respect to the bilinear form (\ref{1.9}). The recursion operator $\mathcal{R}_{12}$ and its adjoint $\mathcal{R}_{12}^{\star}$ are defined by
\begin{equation}
	\mathcal{R}_{12} := (\mathcal{J}_{12}^{(1)})^{-1}\mathcal{J}_{12}^{(2)},\quad
	\mathcal{R}_{12}^{\star} := \mathcal{J}_{12}^{(2)}(\mathcal{J}_{12}^{(1)})^{-1} = i\theta_{12}^{-}T^\delta_{12} - \theta_{12}^{+},\label{1.15}
\end{equation}
and in view of (\ref{1.15}), they satisfy the compatibility condition
\begin{equation*}
	\mathcal{R}_{12}^{\star}\mathcal{J}_{12}^{(1)} = \mathcal{J}_{12}^{(1)}\mathcal{R}_{12}.
\end{equation*}
Consequently, the first few equations of the ILW hierarchy are
\begin{align*}
	&u_t-u_x=0, \quad \text{for}\quad n=1, \\
	&u_t+\frac{1}{\delta}u_x+2uu_x+T^\delta u_{xx}=0,\quad \text{for}\quad n=2;\\
	&u_t-\left(\frac{1}{4}({T^{\delta}}^2-1)u_{xx}+ u^3+\frac{3}{2}(uT^\delta u_x+T^\delta uu_{xx})\right)_x=0, \quad \text{for}\quad n=3.
\end{align*}

In common with the classical KdV and BO equations, between which the ILW equation (\ref{ILW}) serves as a model-theoretical bridge and admits a family of unique exact solitary-wave solutions \cite{AT,SA,SAK} of the form
\begin{align}
	u(x,t)=Q_c(x-ct-x_0), \quad Q_c(s)=\frac{a\sin(a\delta)}{\cosh(as)+\cos(a\delta)},\quad c>0,~ x_0\in \mathbb{R}, \label{1.16}
\end{align}
where $c$ represents the wave speed and $a$ is the unique solution of the transcendental equation
\begin{align*}
	a\delta \cot(a\delta)=1-c\delta, \quad a\in \left(0, \frac{\pi}{\delta}\right).
\end{align*}
Substituting (\ref{1.16}) into equation (\ref{ILW}) yields
\begin{align}
	T^\delta \partial_x Q_c +\left(\frac{1}{\delta}-c\right)Q_c+Q_c^2=0, \quad c>0.\label{1.17}
\end{align}
Beyond these fundamental solitary waves, the ILW equation (\ref{ILW}) also supports more complex solutions, including multi-solitons \cite{JE} that admit parametric representations similar to the single-soliton case \cite{AT,LP,M,PD}.
The ILW $n$-soliton solution $U^{(n)}_\mathbf{c}$ is characterized by a collection of wave speeds $\mathbf{c}=(c_1, c_2, \cdots, c_n)$ and initial positions $\mathbf{x}=(x_1, x_2, \cdots, x_n)$, where the wave speeds satisfy $c_j>0$ with $c_j\neq c_k$ for $j\neq k$ ($j,k=1, 2, \cdots, n$). Moreover, in the long-time limit, these $n$-soliton solutions decompose into a superposition of $n$ individual solitons as follows
\begin{align}
	\lim_{|t| \to +\infty} \left\| U^{(n)}(t,\cdot;\mathbf{c},\mathbf{x}) - \sum_{j=1}^{n} Q_{c_{j}}(\cdot - c_{j}t - x_{j}) \right\|_{H^{s}(\mathbb{R})} = 0, \quad s \in \mathbb{N}.\label{resolution}
\end{align}

In recent years, the stability theory for solitary waves and multi-solitons has emerged as a prominent research direction. Within soliton theory, stability concepts are systematically categorized into four distinct types based on analytical methodology and robustness guarantees: (i) linear (spectral) stability, concerning the eigenvalue distribution of linearized operators; (ii) Lyapunov (dynamical) stability, established via positive definiteness of the second variation of Lyapunov functionals at soliton solutions; (iii) orbital (nonlinear) stability, which requires that solutions remain within a neighborhood of the soliton orbit under finite-amplitude disturbances; and (iv) asymptotic stability, demanding convergence to specific soliton profiles.

As fundamental models in integrable systems, the stability of solitary waves for both KdV and BO equations has been extensively investigated. For instance, Maddocks and Sachs \cite{MS} established orbital stability of KdV multi-solitons in the energy space defined by conservation laws. Killip and Visan \cite{KV} proved orbital stability of KdV multi-solitons in $H^{-1}(\mathbb{R})$ using low-regularity conservation laws in \cite{BP}. For the BO equation, Wang \cite{LW} and Matsuno \cite{Mat} studied dynamical stability of multi-soliton solutions. Additionally, Badreddine, Killip and Visan \cite{BKP} proved that the multi-soliton solutions to the BO equation are uniformly orbitally stable in $H^s(\mathbb{R})$ with $-\frac{1}{2}<s<\frac{1}{2}$.

In contrast to the extensive results for KdV and BO equations, stability analysis for solitary wave solutions of the ILW equation remains comparatively underdeveloped. Existing studies primarily focus on orbital stability. Albert and Bona \cite{AB} employed a rigorous variational approach combined with spectral analysis to establish orbital stability of ILW solitary waves in its energy space $\mathcal{X} = \{g \in L^2(\mathbb{R})| \int_\mathbb{R}(1+\alpha(k))|\hat{g}|^2\mathrm{d}k < \infty\}$. Subsequent work \cite{A} extended these results to spaces $H^{\frac{1}{2}}(\mathbb{R})$, while \cite{PCA} addressed linear and orbital stability of periodic wave solutions in periodic Sobolev spaces.

Considering the growing importance of understanding coherent structures in nonlocal dispersive equations, a systematic analysis of soliton and multi-soliton stability for the ILW equation represents a significant open research direction. Hence, in this work, we aim to establish dynamical stability results of $n$-soliton solutions to the ILW equation (\ref{ILW}) in an appropriate Sobolev space.

Our strategy is to adapt the methods of Maddocks and Sachs \cite{MS}.
In accordance with the ideas in \cite{MS}, we construct an appropriate  Lyapunov functional $\mathcal{S}_{n}$ of the ILW $n$-solitons is given by
\begin{align}
	\mathcal{S}_{n}(u) = H_{n+1}(u) + \sum_{m=1}^{n} \mu_{m} H_{m}(u),\label{lagrange}
\end{align}
and $\mu_{m}$ are Lagrange multipliers which will be expressed in terms of the elementary symmetric functions of $c_{1},c_{2},\dots,c_{n}$. Using (\ref{lagrange}), this condition can be written as the following Euler-Lagrange equation
\begin{align}
	\frac{\delta H_{n+1}(u)}{\delta u} + \sum_{m=1}^{n} \mu_{m} \frac{\delta H_{m}(u)}{\delta u} = 0,\quad \text{at } u = U^{(n)}.\label{1.20}
\end{align}
The dynamical stability of $U^{(n)}$ is implied by the fact that $U^{(n)}(x)$ is a minimizer of the functional $H_{n+1}$ under the following $n$ constraints
\begin{align}
	H_{m}(u) = H_{m}\left(U^{(n)}\right),\quad m = 1,2,\dots,n,
\end{align}
which requires that the self-adjoint second variation of the operator $\mathcal{S}_n$,
\begin{align}
	\mathcal{L}_n:=\mathcal{S}''_{n}(U^{(n)}), \label{VL}
\end{align}
is strictly positive if one modulates the directions given by the constraints.

The demonstration of dynamical stability for the ILW equation relies on the recursion operators and the construction of a Lyapunov functional. This approach utilizes the trace formula derived from the inverse scattering transform (IST) framework to forge a link with conservation laws, in a manner analogous to techniques applied in the study of the Camassa-Holm equation \cite{AYP, CGI, DKS}. The stability proof centers on a spectral analysis of the second variation $\mathcal{L}_{n} = \mathcal{S}_{n}''(U^{(n)})$ evaluated at a smooth $n$-soliton $U^{(n)}$, coupled with an eigenvalue computation for the Hessian matrix $D = \left\{\frac{\partial^2 \mathcal{S}_{n}}{\partial \mu_{i} \partial \mu_{j}} \right\}$.

The spectral characterization of the linearized operator $\mathcal{L}_{n}$ around the $n$-soliton $U^{(n)}$ is central to this endeavor. In the case of a single soliton, the operator $L_n = H''_{n+1}(Q_c) + c H''_{n}(Q_c)$ is analyzed. The inertia index $\mathrm{in}(L_n)$ is determined by constructing the derivative operators $\mathcal{J}L_n$ and $L_n\mathcal{J}$ and leveraging IST theory, specifically employing Jost solutions and square eigenfunctions. For the general $n$-soliton, the recursive operator $\mathcal{R}$ originating from the bi-Hamiltonian structure is applied to compute $\mathrm{in}(\mathcal{L}_n)$.

Finally, we conclude by proving the orbital stability of double solitons. This final result is obtained through a spectral analysis of the linearized operator $\mathcal{L}_{2}$ around the double soliton $U^{(2)}$, supported by a contradictory argument. Our analysis shows that $\mathcal{L}_{2}$ has one simple negative eigenvalue and one double eigenvalue at zero.

In what follows, we will present our main results. The first is the dynamical stability of the multi-solitons to the ILW equation (\ref{ILW}).
\begin{theorem}[\textbf{Dynamical stability of $n$-solitons}]\label{thm1.1}
Given $ n\in \mathbb{N} $, $ n\geq 1 $, a collection of wave speeds $\textbf{c} = (c_{1},\dots,c_{n})$ with $ 0 < c_{1} < \dots < c_{n} $ and a collection of space transitions $\textbf{x} = (x_{1},\dots,x_{n})\in \mathbb{R}^{n}$, let $ U^{(n)}(\cdot,\cdot;\textbf{c},\textbf{x}) $ be the corresponding $n$-soliton solutions of (\ref{ILW}). Then for any $ \epsilon > 0 $, there exists $ \delta > 0 $ such that for any $ u_{0}\in H^{\frac{n}{2}}(\mathbb{R}) $, the following stability property holds, if

\begin{align}
	&\bigl\lVert u_{0} - U^{(n)}(0,\cdot;\textbf{c},\textbf{x}) \bigr\rVert_{H^{\frac{n}{2}}} < \delta,
\end{align}
then for any $t \in \mathbb{R}$, the corresponding solution of (\ref{ILW}) verifies
\begin{align}
\inf_{\tau \in \mathbb{R},\  \textbf{y} \in \mathbb{R}^{n}} \bigl\lVert u(t) - U^{(n)}(\tau,\cdot;\textbf{c},\textbf{y}) \bigr\rVert_{H^{\frac{n}{2}}} < \epsilon.
\end{align}
\end{theorem}
Since the stability result in \cite{MS} coincides with orbital stability theorem in \cite{GSS} for $n=1$, we obtain the following corollary regarding the orbital stability of ILW soliton solutions.
\begin{cor}
The single soliton solution of the ILW equation is orbital stable in $H^{\frac{1}{2}}(\mathbb{R})$.
\end{cor}
\begin{re}
For $n \geq 2$, it should be noted that the stability result in \cite{MS} aligns with the notion of orbital stability in \cite{GSS} when the stability framework of \cite{MS} is suitably extended. In Hamiltonian systems, invariant functionals (i.e., integrals of motion) generate Hamiltonian flows that commute with the original time evolution.
\end{re}

Then we will provide a more precise description of the orbital stability of the double solitons to the ILW equation.

\begin{theorem}[\textbf{Orbital stability of double solitons}]\label{thm1.2}
	The double solitons $U^{(2)}_{c_1,c_2}(t,x;x_1,x_2)$ of ILW equation (\ref{ILW}) with $0<c_1<c_2$ are orbitally stable in $H^1(\mathbb{R})$ in the following sense.
There exist parameters $\epsilon_0$ and $A_0$, depending on $c_1$ and $c_2$, if there exists $\epsilon \in (0,\epsilon_0)$ such that for any $u_0\in H^1(\mathbb{R})$,
\begin{align}\label{u02}
	\bigl\lVert u_0 - U^{(2)}_{c_1,c_2}(0;0,0) \bigr\rVert_{H^1(\mathbb{R})} < \epsilon,
\end{align}
then there exist $x_1(t), x_2(t) \in \mathbb{R}$, such that the corresponding solution $u(t,x)$ of the ILW equation (\ref{ILW}) with the initial data $u(0)=u_0$ satisfies $u_t \in C([0,\infty),H^1(\mathbb{R}))$ and
\begin{align}\label{1.22}
	\sup_{t\in (0,\infty)}|\!|u(t)-U^{(2)}_{c_1,c_2}(t,x;x_1,x_2)|\!|_{H^1(\mathbb{R})}<A_0\epsilon,
\end{align}
with
\begin{align}
	\sup_{t\in (0,\infty)}(|x_1'(t)|+|x_2'(t)|)\leq CA_0\epsilon.
\end{align}
\end{theorem}

Moreover, something is interesting is that one can express the negative eigenvalues of the isoinertial operator $\mathcal{L}_n$ in terms of the wave speeds $\{c_j\}^n_{j=1}$ as follows.
\begin{theorem}[\textbf{The negative eigenvalues of $\mathcal{L}_n$}]\label{th1.3}
	The linearized operator $\mathcal{L}_{n}$ about the $n$-soliton solution possesses exactly $\left\lfloor\frac{n+1}{2}\right\rfloor$ negative eigenvalues $\lambda_{k}$, where $k=1,2,\cdots,\left\lfloor\frac{n+1}{2}\right\rfloor$ and $\lfloor x\rfloor$ denotes the floor function. Furthermore, the negative eigenvalues of $\mathcal{L}_n(n\geq 1$)  have the following expression
	\begin{align}\label{negetive}
		\lambda_{k} = -M^\delta_k (c_{2k-1}-\frac{1}{\delta})\prod_{ j\neq 2k-1}^{n}(c_{j}-c_{2k-1}),\quad k=1,2,\cdots,\left\lfloor\frac{n+1}{2}\right\rfloor,
	\end{align}
	where $M^\delta_k$ are positive constants independent of the wave speeds $c_1,\cdots, c_n$.
\end{theorem}
\begin{re}
	Since the ILW equation serves as the bridge connecting KdV equation and BO equation. The stability results above are also valid for both equations. The specific results can be found in \cite{WA} and \cite{LW}, respectively.
\end{re}

The proof of our main stability result, Theorem \ref{thm1.1}, proceeds by transforming the stability problem into a spectral analysis of an explicitly constructed linearized operator. Through the IST method, we establish a fundamental connection between the Hamiltonians and scattering data. Similarly, the proof of Theorem \ref{thm1.2} on orbital stability is reduced to a detailed spectral analysis of the linearized operator $\mathcal{L}_n$ arising from the second variation of the Lyapunov functional. Finally, the proof of Theorem \ref{th1.3} relies on the application of spectral analysis results. The proof strategy unfolds as follows.

In Section \ref{sec.2}, we begin by establishing the properties of the singular integral operators $T^\delta$ and $T^\delta_{12}$. We then introduce the IST framework for the ILW equation, from which the $n$-soliton solutions are formally derived.

Section \ref{sec.3} is devoted to a detailed spectral analysis of the linearized operator $\mathcal{L}_n = \mathcal{S}_n''(U^{(n)})$. We demonstrate that its inertia index is $\left( \left\lfloor \frac{n+1}{2} \right\rfloor, n \right)$, thereby satisfying the crucial stability condition.

In Section \ref{sec.4}, we apply the framework of constrained variational principles to establish the dynamical stability result in Theorem \ref{thm1.1}. We then prove Theorem \ref{thm1.2} on orbital stability of double solitons ($n=2$) through a refined spectral analysis of the linearized operator $\mathcal{L}_2$, yielding a direct proof of orbital stability. Finally, we establish Theorem \ref{th1.3} based on the spectral analysis results from Section \ref{sec.3}.

\section{Background results for the ILW equation}\label{sec.2}
	In this section, we present fundamental results from the inverse scattering transform theory for the ILW equation (\ref{ILW}).
\subsection{Properties of the operator $T^\delta$ and $T^\delta_{12}$}\label{sub2.1}
We first systematize key analytical properties of the fundamental singular integral operator $T^\delta$ and its bilinear extension $T^\delta_{12}$, for more details can refer to \cite{S1,Sa}.

 For the operator $T^\delta$, it follows the commutation relation
\begin{equation}
	T^\delta \partial_x f = \partial_x T^\delta f
\end{equation}
holds for $f \in H^1(\mathbb{R})$. The operator $T^\delta$ is skew-adjoint satisfying
\begin{equation}
	\langle T^\delta f, g \rangle = -\langle f, T^\delta g \rangle,
\end{equation}
and maps even functions to odd functions. The extended operator $T^\delta_{12}$ admits two equivalent representations. The first one is defined in (\ref{1.13}) and the second can be presented as
\begin{equation}
	T^\delta_{12} = i (1 + e_1 e_2) (1 - e_1 e_2)^{-1},
\end{equation}
where $(e_j f)(x_j) = f(x_j + i\delta)$ is the displacement operator. Moreover, for $f_{12} \in L^2(\mathbb{R}^2, \mathbb{C})$, we can know the following properties
\begin{equation}
	T^{\delta *}_{12} = -T^\delta_{12}, \quad \partial_{x_j} T^\delta_{12} f_{12} = T^\delta_{12} \partial_{x_j} f_{12},  \ j=1,2.
\end{equation}
and for any $a \in L^2(\mathbb{R})$, we have
\begin{equation}
	T^\delta_{12} a(x_j) = T^\delta_j a(x_j).
\end{equation}

If we defining the projections
\begin{equation}
	f_{12}^{(\pm)} = \pm \frac{1}{2} (1 \mp i T^\delta_{12}) f_{12},
\end{equation}
then $f_{12}^{(+)}$ and $f_{12}^{(-)}$ are holomorphic in $\operatorname{Im}(x_1 + x_2) > 0$  and $\operatorname{Im}(x_1 + x_2) < 0$, repectively. Moreover, one has
\begin{equation}
	T^\delta_{12} (f_{12}^{(+)} - f_{12}^{(-)}) = i (f_{12}^{(+)} + f_{12}^{(-)}).
\end{equation}
 When in the deep-water limit $\delta \to \infty$, $T_{12}$ reduces to the generalized Hilbert transform referred to (\ref{BO}), namely
\begin{equation}
	\lim_{\delta \to \infty} (T^\delta_{12} f)(x_1,x_2) = (H_{12} f)(x_1,x_2) = \frac{1}{\pi} \int_{\mathbb{R}} \frac{F(y, x_1 - x_2)}{y - (x_1 + x_2)}  dy.
\end{equation}

\subsection{Eigenvalue problem and conservation laws}
Here we list some results related to the theory of the inverse scattering transform ILW equation. For detailed proof can refer to \cite{K,KAS,KSA}.

The IST scheme for ILW equation (\ref{ILW}) is given by
\begin{align}
	&i\psi_x^++(u-\lambda)\psi^+=\mu\psi^-,\label{lax1}\\
	&i\psi_t^{\pm}+2i(\lambda+\frac{1}{2\delta})\psi_x^{\pm}+\psi^+_{xx}+[\mp iu_x-T^\delta u_x+v]\psi^{\pm}=0,\label{lax2}
\end{align}
where $\lambda$ and $\mu$ are constants given by $\lambda=-k\coth 2k\delta, \mu=k\text{cosech} 2k\delta$, and $v$ is a constant determined by fixing the Jost functions of (\ref{lax1}). Functions $\psi^{\pm}(x)$ denote boundary values
\[
\psi^{\pm}(x) = \lim_{\operatorname{Im} z \to 0^\pm} \psi^{\pm}(z)
\]
analytic in horizontal strips $0 < |\operatorname{Im} z| < 2\delta$ with vertical periodicity. In order to analyse the scattering problem of (\ref{lax1}), one can define $W(x,k)=\psi(x,k)e^{ikx}$, where $\psi(x,k)=\psi^{\pm}(x\pm i\delta,k)$ and $W(x,k)=W^{\pm}(x\pm i\delta,k)$, then the spectral problem (\ref{lax1}) and (\ref{lax2}) become
\begin{align}
	&iW_{x}^{+} + \left(\zeta_{+} + \frac{1}{2\delta}\right)(W^{+} - W^{-}) = -uW^{+}, \label{lax3}\\
	&iW_{t}^{\pm} - 2i\zeta_{+}W_{x}^{\pm} + W_{xx}^{\pm} + \left[\mp iu_{x} - Tu_{x} + \rho\right] W^{\pm} = 0, \label{lax4}
\end{align}
where $\rho = -2k\zeta_{+} + k^{2} + v$, $\zeta_{\pm} = \zeta_{\pm}(k) = k \pm (k\operatorname{coth} 2k\delta - \frac{1}{2\delta})$.

For real $k$, when $|x|\to\infty $, the Jost solutions $N(x,k), \bar{N}(x,k), M(x,k), \bar{M}(x,k)$are defined as the solution to (\ref{lax3}) with the boundary conditions
\begin{align}\label{2.13}
	\left.
	\begin{array}{l}
		M(x; k) \rightarrow 1, \\
		\bar{M}(x; k) \rightarrow e^{2ikx},
	\end{array}
	\right\}
	\text{ as } x \rightarrow -\infty, \quad
	\left.
	\begin{array}{l}
	N(x; k) \rightarrow e^{2ikx}, \\
		\bar{N}(x; k) \rightarrow 1 ,
	\end{array}
	\right\}
		\text{ as } x \rightarrow \infty,
\end{align}

Furthermore, through the analysis, there is the following relationship between the Jost solutions represented by the scattering data
\begin{align}
	M(x,k)&=a(k)\bar{N}(x,k)+b(k)N(x,k), \label{2.14}\\
	N(x,k)&=\bar{a}(k)\bar{M}(x,k)+\bar{b}(k)M(x,k),\label{2.15}
\end{align}
where the functions $a(k), b(k), \bar{a}(k), \bar{b}(k)$ are all independent of variable $x$ and referred to as scattering data, possessing the following properties
\begin{align*}
	&\bar{a}(k) = a^{*}(-k), \\
	&\bar{b}(k) = -b(-k) = -\left(\frac{d\zeta_{+}}{d k}\right)\left(\frac{d\zeta_{-}}{d k}\right)^{-1} b^{*}(k), \\
	&|a(k)|^{2} -\left(\frac{d\zeta_{+}}{d k}\right)\left(\frac{d\zeta_{-}}{d k}\right)^{-1}|b(k)|^{2} = 1.
\end{align*}

By virtue of the boundary condition (\ref{2.13}) and taking the constant $\rho$ to be zero. From (\ref{lax4}), we obtain
\begin{align*}
	&a(k,t) = a(k,0), \\
	&b(k,t) = b(k,0) \exp \left[ -4ik \left( \lambda + \frac{1}{2\delta} \right) t \right], \\
	&b_{l}(t) = b_{l}(0) \exp \left[ 4\kappa_{l} \left( \lambda_{l} + \frac{1}{2\delta} \right) t \right],
\end{align*}
where $\lambda_{l} = \lambda(i\kappa_{l}) = -\kappa_{l} \cot 2\kappa_{l}\delta$, and $k_l=i\kappa_l, 0<\kappa_l<\frac{\pi}{2\delta}$ is the only column of simple zeroes of $a(k)$. Based on the above facts, the scattering datas of spectral problem (\ref{lax3}) and (\ref{lax4}) are now given by
\begin{align}
		\mathcal{S}= \{ a(k), b(k), \bar{a}(k), \bar{b}(k), \kappa_{l}, l=1,2,\cdots,N\}.
\end{align}
In particular, when $u$ is a soliton potential given by (\ref{1.16}), one has that $|a(k)|\to 1$ and
\begin{align}
	a=2\kappa_1, \quad c_1=2\lambda_1+\delta^{-1}.
\end{align}

The conservation laws for the ILW equation emerge naturally from the analytic properties of the scattering coefficient \(a(k)\). Its explicit factorization reveals the underlying structure of the conservation laws through the decomposition
\begin{align}
	a(k) = \hat{a}(k) \prod_{l=1}^{N} \frac{\zeta_{+} - \zeta_{+l}}{\zeta_{+} + \zeta_{-l}},
\end{align}
where \(\hat{a}(k)\) is a zero-free function satisfying \(\lim_{|\zeta_{+}|\to\infty} \hat{a}(k) = 1\), and \(\zeta_{\pm l} = \zeta_{\pm}(i\kappa_l)\) denote the discrete eigenvalues associated with soliton solutions.

The asymptotic expansion of the logarithmic function \(\ln a(k)\) via contour integration in the complex \(\zeta_{+}\)-plane generates the essential series representation
\begin{align}
	\ln a(k) = \frac{1}{2\delta} \sum_{n=1}^{\infty} \frac{1}{(2\zeta_{+})^n} \left[ \frac{2^n}{\pi} \int_{-\infty}^{\infty} (\zeta_{+}')^{n-1} \ln |a(k')|^2  d\zeta_{+}' + \frac{2^{n+1} i \delta}{n} \sum_{l=1}^{N} \left( \zeta_{+l}^{n} - \zeta_{+l}^{*n} \right) \right].
\end{align}
The coefficients of this expansion directly yield the infinite sequence of conserved quantities through the trace formula
\begin{align}
	I_n = -\frac{2^n \delta}{\pi} \int_{-\infty}^{\infty} \zeta_{+}^{n-1} \ln |a(k)|^2  d\zeta_{+} - \frac{2^{n+1} i \delta}{n} \sum_{l=1}^{N} \left( \zeta_{+l}^{n} - \zeta_{+l}^{*n} \right). \label{2.20}
\end{align}

 In terms of $I_n$, the conservation laws $H_n$ presented in Section1 can be expressed as follows
\begin{align}
	H_n=I_n, \quad \text{for all}~n\in \mathbb{N}.
\end{align}

The first two conserved quantities possess particularly transparent physical interpretations. For \(n=0\), the mass conservation law exhibits explicit depth-dependent scaling
\begin{align}
	H_0=\int_{-\infty}^{\infty} u  dx = 4 \delta\sum_{l=1}^{N} \kappa_l  - \frac{\delta}{\pi} \int_{-\infty}^{\infty} \ln |a(k)|^2  d\zeta_{+},
\end{align}
and for $n=1$,
\begin{align}
	H_1=\frac{1}{2}\int_{-\infty}^{\infty} u^2  dx = 4\delta \sum_{l=1}^{N} \kappa_l c_l + \frac{\delta}{\pi} \int_{-\infty}^{\infty} \zeta_{+} \ln |a(k)|^2  d\zeta_{+}
\end{align}
 will make a significant contribution to the subsequent calculations.

Moreover, similar to the KdV and BO equation case, the ILW conservation laws are in involution, namely,  $ H_{n} $($ n=0,1,2,\dots $) commute with each other in the following Poisson bracket
\begin{align*}
	\int_{-\infty}^{\infty}\left(\frac{\delta H_{n}}{\delta u}(x)\right)\bigg|_{u=U^{(m)}}\frac{\partial}{\partial x}\left(\frac{\delta H_{l}}{\delta u}(x)\right)\bigg|_{u=U^{(m)}}\text{d}x=0,\quad n,l=0,1,2,\dots
\end{align*}
Note that $ H_{0} $ is the unique Casimir function of Eq. (\ref{ILW}).

\subsection{Variational characterization of the bound state profiles}
 Recall that the Lyapunov functional $\mathcal{S}_n$ of the ILW $n$-solitons is given by
\begin{align*}
	\mathcal{S}_n(u)=H_{n+1}(u)+\sum_{m=1}^{n}\mu_m H_m(u),
\end{align*}
and $\mu_m, m=0,1,\cdots,n$ are Lagrange multipliers, then the ILW $n$-solitons satisfy the
Euler-Lagrange equation. This provides a variational characterization of $U^{(n)}$. In order to show the explicit value of the Lagrange multipliers $\mu_m$, It is necessary for us to consider components of the ILW conservation quantities  $H_m$ which related to $1$-soliton profile $Q_c$. Since (\ref{1.17}) imply that $1$-soliton with speed $c$ satisfies the following variational principle
\begin{align}
	\frac{\delta}{\delta u} \left( H_{j+1}(u) + c H_j(u) \right) = 0, \quad j \in \mathbb{Z}_+,
\end{align}
that is
\begin{align}
	\left(H_{j+1}^{\prime}(Q_c) + c H_j^{\prime}(Q_c) \right) = 0.\label{2.25}
\end{align}
 Now multiply (\ref{2.25}) with $\frac{\text{d}Q_c}{\text{d}c}$, then for each $n$, we have
 \begin{align}
 	\frac{\mathrm{d} H_{n+1}(Q_c)}{\mathrm{d} c} = (-c)^n \frac{\mathrm{d} H_1(Q_c)}{\mathrm{d} c}.
 \end{align}
 and therefore
 \begin{align}
 	H_{n+1}\left( Q_{c} \right) =(-1)^n \int_{0}^{c} y^{n} \frac{\mathrm{d} H_{1} \left( Q_{y} \right)}{\mathrm{d} y}  \mathrm{d} y
 \end{align}
 Taking account of the fact that the reflection coefficient $\ln|a(k)|$ of the trace identity (\ref{2.20}) becomes zero for $u=U^{(m)}$ with the associated
 discreet eigenvalue $i\kappa_n$ for $n=1,2, \cdots,m$ and $0<\kappa_n<\frac{\delta}{2\pi}$. In particular, we can derive from (\ref{2.20}) the formula
 \begin{align}
 	H_j(u)=- \frac{2^{j+1} i \delta}{n} \left( \zeta_{+}^{j} - \zeta_{+}^{*j} \right)\in \mathbb{R}.
 \end{align}
 More precisely, one has the following for the conservation laws $H_1$
 \begin{align}
 	H_1(Q_c)=4\delta \kappa c.
 \end{align}
 Particularly, the derivative of $H_1(Q_c)$ with respect to the wave speed
$c$ can be computed in the following form
\begin{align}\label{2.30}
	\frac{\mathrm{d} H_{1}(Q_{c})}{\mathrm{d} c} =4\delta \left(\frac{c\sin^22\kappa\delta}{4\kappa\delta-\sin 4\kappa\delta}+\kappa\right)\stackrel{\Delta}{=}\mathcal{G}^\delta_{\kappa,c}>0.
\end{align}

Base upon facts one have the follow results.
\begin{prop}
	The profiles of the ILW $n$-solitons $U^{(n)}$ satisfy (\ref{lagrange}) if the Lagrange multipliers $\mu_{m}$ are symmetric functions of the wave speeds $c_{1},c_{2},\cdots,c_{n}$ which satisfy the following
	\begin{align}
		\prod_{m=1}^{n}(x + c_{m}) = x^{n} + \sum_{m=1}^{n}\mu_{m}x^{n-m},\quad x\in \mathbb{R}.
	\end{align}
	In particular, $\mu_{m}$ are given by the following Vieta's formulas: for $k=1,\ldots,n$
	\begin{align}
		\mu_{m+1-k} = \sum_{1\leq i_{1} < \cdots < i_{k} \leq n} \biggl( \prod_{j=1}^{k}c_{i_{j}} \biggr).\label{Vieta}
	\end{align}
\end{prop}
\begin{proof}
With the Euler-Lagrange equation (\ref{1.20}) and taking $t \to \infty$, the asymptotic decomposition yields
	\begin{align}
		\lim_{t\to\infty} \sum_{j=1}^n \left[ \frac{\delta H_{n+1}}{\delta u}(Q_{c_j}) + \sum_{m=1}^n \mu_m \frac{\delta H_m}{\delta u}(Q_{c_j}) \right] = 0.
	\end{align}
Hence, for each single soliton $Q_c$, the recurrence relation holds
	\begin{align}
		\frac{\delta H_{k+1}}{\delta u}(Q_c) + c \frac{\delta H_k}{\delta u}(Q_c) = 0, \quad \forall k \geq 1.
	\end{align}
Therefore, we have the following reduction
	\begin{align}
		\frac{\delta H_{m+1}}{\delta u}(Q_c) = (-c)^m \frac{\delta H_1}{\delta u}(Q_c).
	\end{align}
	This non-degeneracy condition ensures that the subsequent derivation is well-defined and non-trivial. Therefore we obtain the fundamental equation system
	\begin{align}
		\sum_{j=1}^n \left[ (-c_j)^n + \sum_{m=1}^n \mu_m (-c_j)^{m-1} \right] Q_{c_j} = 0,
	\end{align}
it implies that for each $j$, we have
	\begin{align}
		(-c_j)^n + \sum_{m=1}^n \mu_m (-c_j)^{m-1} = 0.
	\end{align}
which implies that $-c_j$ are the roots of the polynomial $ x^{n} + \sum_{m=1}^{n}\mu_{m}x^{n-m}$. Since $0<c_1<c_2<\cdots<c_n$, we
obtain(\ref{Vieta}) from Vieta's formula immediately.
\end{proof}
\subsection{Bi-Hamiltonian formation}
The Hamiltonian formulation of the ILW equation, derived from (\ref{1.14}), enables the construction of a recurrence operator characterizing conserved quantities. Consider conservation laws $H_m(u): H^{\frac{m-1}{2}}(\mathbb{R}) \to \mathbb{R}$ where $m \in \mathbb{N}$. The recurrence operator $\mathcal{R}(u)$ satisfies the variational relation
\begin{align}
	\frac{\delta H_{m+1}(u)}{\delta u} &= \mathcal{R}(u)\frac{\delta H_m(u)}{\delta u},\label{2.39}
\end{align}
Differing from integrable systems like KdV, the recursion operator $\mathcal{R}(u)$ remains implicitly defined through (\ref{1.15}). Its adjoint $\mathcal{R}^\star(u)$ conforms to the duality condition
\begin{align}
	\mathcal{R}^\star(u) &= \mathcal{J}\mathcal{R} (u) \mathcal{J}^{-1},\label{2.40}
\end{align}
and it is not difficult to see that the operators $\mathcal{R} (u)$ and $\mathcal{R}^\star(u)$ satisfy
\begin{align}
	\mathcal{R}^\star(u) \mathcal{J} &= \mathcal{J}\mathcal{R} (u).\label{2.41}
\end{align}
The above definitions of recursion operators are reasonable since \( \mathcal{R}(u) \) maps the variational derivative of conservation laws of (\ref{ILW}) onto the variational derivative of conservation laws, and \( \mathcal{R}^{\star}(u) \) maps infinitesimal generators of symmetries of (\ref{ILW}) onto infinitesimal generators of symmetries of (\ref{ILW}) is \( u_{x} \) \cite{S1,Sa}, therefore (\ref{2.40}) is well-defined since

\begin{align*}
	(\mathcal{R}^\star(u))^m u_x &= \mathcal{J} (\mathcal{R} (u))^m H_1'(u) = \mathcal{J} (\mathcal{R} (u))^m u, \quad m \in \mathbb{N}.
\end{align*}
For subsequent applications, establishing the uniqueness and differentiability of $\mathcal{R}(u)$ becomes essential. In contrast, the Korteweg-de Vries system ($\delta = 3$) exhibits explicit structure. For Schwartz-class functions $u \in \mathcal{S}(\mathbb{R})$, its recursion operator is $\mathcal{R}_K(u) = -\partial_x^2 - \tfrac{2}{3}u - \tfrac{2}{3}\partial_x^{-1} u \partial_x$, then $\mathcal{R}_K'(u) = -\tfrac{2}{3} - \tfrac{2}{3}\partial_x^{-1}(\cdot\partial_x)$.
\begin{prop}\label{prop2.2}
	Given $u \in H^{k+1}(\mathbb{R})$ with $k \geq 0$, there exists a unique linear operator
	\[
	\mathcal{R}(u): H^{k+1}(\mathbb{R}) \to H^{k}(\mathbb{R}),
	\]
	such that relations (\ref{2.39}) and (\ref{2.41}) hold true. Moreover, $\mathcal{R}(u)$ is differentiable with respect to $u$.
\end{prop}
\begin{proof}
	The fundamental approach involves establishing a connection between the recursion operators $\mathcal{R}(u)$ and $\mathcal{R}_{12}$ (\ref{1.15}). Assuming $u \in \mathcal{S}(\mathbb{R})$, it follows from (\ref{1.14}) and (\ref{2.39}) that
	\begin{align*}
		\mathcal{S}(\mathbb{R}) \ni H_{m+1}^{\prime}(u) & = \frac{i}{2^{m+1}}\mathcal{J}^{-1}\int_{\mathbb{R}} \delta(x_{1} - x_{2}) \theta_{12}^{-} \mathcal{R}_{12}^{m+1} \cdot 1  \mathrm{d}x_{2} \\
		& = \mathcal{R}(u)H_{m}^{\prime}(u) \\
		& = \mathcal{R}(u) \frac{i}{2^m}\mathcal{J}^{-1}\int_{\mathbb{R}} \delta(x_{1} - x_{2}) \theta_{12}^{-} \mathcal{R}_{12}^{m} \cdot 1  \mathrm{d}x_{2}.
	\end{align*}
	The uniqueness of $\mathcal{R}(u)$ follows by an induction argument over $m$. Furthermore, the asymptotic behavior of the operator is characterized by
	\[
	\mathcal{R}(u) \sim -T^\delta\partial_{x} + L(u),
	\]
	where $L(u): \mathcal{S}(\mathbb{R}) \mapsto \mathcal{S}(\mathbb{R})$ represents a differentiable higher-order remainder term. By standard density argument, $\mathcal{R}(u)$ itself is differentiable and satisfies $\mathcal{R}^{\prime}(u) \sim L^{\prime}(u)$.
\end{proof}
\begin{re}
The propositions above will be shown in Section \ref{sec.3} that understanding the spectral information of the (adjoint) recursion operators $\mathcal{R} (u)$ and $\mathcal{R}^\star(u)$ is essential in proving the (spectral) stability of the ILW multi-solitons.
\end{re}

We first observe that the differential equation (\ref{1.17}) verified by the soliton profile and the bi-Hamiltonian structure (\ref{2.39}) imply that the 1-soliton $Q_{c}(x-ct-x_{0})$ with speed $c>0$ satisfies, for all $m\geq 2$ and for any $t\in\mathbb{R}$, the following variational principle

\begin{align}\label{recursion}
	H_{m+1}^{\prime}(Q_{c}) + cH_{m}^{\prime}(Q_{c}) = \cdots = \mathcal{R}^{m-1}(Q_{c})(H_{2}^{\prime}(Q_{c}) + cH_{1}^{\prime}(Q_{c})) = 0,
\end{align}
and it holds true since the functions $H_{m}^{\prime}(Q_{c}) + cH_{m-1}^{\prime}(Q_{c}) \in H^{1}(\mathbb{R})$ belong to the domain of $\mathcal{R}(Q_{c})$.

Next, let us recall that the soliton $Q_{c}(x-ct-x_{0})$ (\ref{1.16}) is a solution of the ILW equation. For simplicity, we denote $Q_{c}$ by $Q$. Then by (\ref{2.39}), we have
\begin{align}
	H'_{m+1}(Q) &= \mathcal{R}(Q)H'_{m}(Q).
\end{align}

To analyze the second variation of the actions, we linearize the equation (\ref{2.39}) by letting $u = Q + \varepsilon z$, and obtain a relation between linearized operators $H''_{m+1}(Q) + cH''_{m}(Q)$ and $H''_{m}(Q) + cH''_{m-1}(Q)$ for all $m \geq 2$.

\begin{prop}\label{prop2.3}
	Suppose that $Q$ is a soliton profile of the BO equation with speed $c > 0$. If $z \in H^{m}(\mathbb{R})$ for $n \geq 1$, then there holds the following iterative operator identity
	\begin{align}
		(H''_{m+1}(Q) + cH''_{m}(Q))z &= \mathcal{R}(Q)(H''_{m}(Q) + cH''_{m-1}(Q))z.\label{2.43}
	\end{align}
\end{prop}

\begin{proof}
	Let $u = Q + \varepsilon z$. By (\ref{2.39}) and the definition of G$\hat{a}$teaux derivative, we have
	\begin{align}
		H''_{m+1}(Q)z &= \mathcal{R}(Q)(H''_{m}(Q)z) + (\mathcal{R}'(Q)z)(H'_{m}(Q)). \label{2.44}
	\end{align}
	Then by (\ref{2.44}),
	\begin{align*}
		(H''_{m+1}(Q) + cH''_{m}(Q))z
		&= \mathcal{R}(Q)\left((H''_{m}(Q) + cH''_{m-1}(Q))z\right) \\
		&\quad + (\mathcal{R}'(Q)z)(H'_{m}(Q) + cH'_{m-1}(Q))
	\end{align*}
	Notice that from Proposition \ref{prop2.2}, $\mathcal{R}'(Q)$ is well-defined. Then (\ref{2.43}) follows directly from (\ref{recursion}).
\end{proof}
\section{Spectral analysis}\label{sec.3}
Let $U^{(n)}(t,x)$ denote the $n$-soliton solution of the ILW equation and set $U^{(n)}(x)=U^{(n)}(0,x)$. We use subscripts $od$ and $ev$ for spaces of odd and even functions, respectively. This section presents a spectral analysis of the linearized operator $\mathcal{L}_n$ around $U^{(n)}$, defined formally in (\ref{VL}). The analysis uses the recursion operators and their adjoints from Section \ref{sec.2}.

We study the spectrum of $\mathcal{L}_n$ via the iso-inertial operator family approach, following \cite{LW,MS, WL,WTLW,ZL}. The spectral structure is derived from two main observations.
First, a form of iso-spectral property holds: the inertia of $\mathcal{L}_n$ (the number of negative eigenvalues and the kernel dimension) is invariant under the time evolution of $U^{(n)}(t,x)$.

Second, for large times, $\mathcal{L}_n(t)$ effectively decouples into a composition of linearized operators around each individual soliton, because the solitons separate due to different speeds. Hence, the spectrum of $\mathcal{L}_n(t)$ converges to the union of the spectra of the single-soliton linearized operators.

Moreover, the $n$-soliton solution $U^{(n)}(t,x)$ decomposes into $n$ well-separated single-soliton solutions $Q_{c_j}(x-c_j t - x_j)$ as described in equation (\ref{resolution}). Consequently, as $t \to \infty$, the spectrum $\sigma(\mathcal{L}_n(t))$ converges to the union of the spectra of the operators $L_{n,j} := H^{\prime\prime}(Q_{c_j})$, that is
\[
\lim_{t \to \infty} \sigma(\mathcal{L}_n(t)) = \bigcup_{j=1}^n \sigma(L_{n,j}).
\]

Hence, in this section we demonstrate a significant result concerning the spectral properties of the linearized operator. Specifically, we show that the inertia of $\mathcal{L}_n$ related to the $n$-soliton solution $U^{(n)}$ consists of exactly $\left\lfloor \frac{n+1}{2} \right\rfloor$ negative eigenvalues, while the dimension of its null space is precisely $n$. Formally, we express this as
\[
\text{in}(\mathcal{L}_n(t)) = \left( \left\lfloor \frac{n+1}{2} \right\rfloor, n \right).
\]
This fundamental result is derived from the inertia properties of the operators $L_{n,j}:=I''_n(Q_{c_j})$, which exhibit the following characteristic behaviors depending on the index parity.
\begin{itemize}
	\item[(i)] For \textbf{odd indices} $j = 2k - 1$: $\text{in}(L_{n,j}) = (1,1)$, signifying that $L_{n,2k-1}$ has exactly one negative eigenvalue.
	\item[(ii)] For \textbf{even indices} $j = 2k$: $\text{in}(L_{n,j}) = (0,1)$, indicating that $L_{n,2k}$ is non-negative.
\end{itemize}

Considering the explicit form of $L_{n,j}$, we recognize it as the summation
\[
L_{n,j} = \sum_{m=1}^n \left( H_{m+1}^{\prime\prime}(Q_{c_j}) + c_j H_m^{\prime\prime}(Q_{c_j}) \right).
\]
Remarkably, as established in Proposition \ref{prop2.3}, this operator admits an elegant factorization
\begin{align}\label{3.1}
	L_{n,j} = \left( \prod_{\substack{k=1 \\ k \neq j}}^n (\mathcal{R}(Q_{c_j}) + c_k) \right) \left( H_2^{\prime\prime}(Q_{c_j}) + c_j H_1^{\prime\prime}(Q_{c_j}) \right) .
\end{align}
Here the coefficients $\sigma_{j,k}$ represent the elementary symmetric functions of the wave speeds $c_1, \dots, c_{j-1}, c_{j+1}, \dots, c_n$.
This factorization reveals the intricate mathematical structure underlying the spectral properties of multi-soliton solutions and provides a powerful tool for analyzing their stability characteristics.

To obtain the spectrum of the operator $L_{n,j}$ (\ref{3.1}), let us consider the spectral analysis of the linearized Hamiltonian
\begin{align}
	L_{m} &:= H_{m+1}^{\prime\prime}(Q) + cH_{m}^{\prime\prime}(Q),\label{3.2}
\end{align}
for all integers $m\geq 1$.  Here we write for simplicity $Q_c$ by $Q$ in the rest of this section. One of crucial ingredients to deal with this spectrum problem is the following operator identities related to the recursion operator $\mathcal{R}(Q)$ and the adjoint recursion operator $\mathcal{R}^{\star}(Q)$ (see (\ref{2.40})).
\begin{lemma}\label{lem3.1}
	The recursion operator $\mathcal{R}(Q)$, the adjoint recursion operator $\mathcal{R}^{\star}(Q)$ and the linearized operator $L_{m}$ for all integers $m \geq 1$ satisfy the following operator identities
	
	\begin{align}
		L_{m}\mathcal{J}\mathcal{R}(Q) &= \mathcal{R}(Q)L_{m}\mathcal{J}, \label{3.3} \\
		\mathcal{J}L_{m}\mathcal{R}^{\star}(Q) &= \mathcal{R}^{\star}(Q)\mathcal{J}L_{m},\label{3.4}
	\end{align}
	where $\mathcal{J}$ is the operator $\partial_{x}$.
\end{lemma}

\begin{proof}
	We need only to prove (\ref{3.4}), since one takes the adjoint operation on (\ref{3.4}) to have (\ref{3.3}). Notice that from Proposition \ref{prop2.3}, one has that the operator $\mathcal{R}(Q)L_{m}=L_{m+1}$ is self-adjoint. This implies that
	\begin{align*}
		(\mathcal{R}(Q)L_{m})^{\star} &= \mathcal{R}(Q)L_{m} = L_{m}\mathcal{R}^{\star}(Q).
	\end{align*}
	On the other hand, in view of (\ref{2.41}), one has
	\begin{align*}
		\mathcal{J}L_{m}\mathcal{R}^{\star}(Q) &= \mathcal{J}\mathcal{R}(Q)L_{m} = \mathcal{R}^{\star}(Q)\mathcal{J}L_{m},
	\end{align*}
	as the advertised result in the lemma.
\end{proof}

\begin{re}
	Types of (\ref{3.3}) and (\ref{3.4}) hold for any solutions of the ILW equation.
\end{re}
In particular, we have the following results regarding ILW $n$-soliton profile $U^{(n)}$.	
\begin{coro}
	let $U^{(n)}$ be the ILW $n$-soliton profile and $\mathcal{L}_{n}$ be the second variation operator defined in (\ref{VL}). Then it is easy to verify that (similar to Lemma \ref{lem3.1}) the following operator identities hold true
\begin{align}
	\mathcal{L}_{n}\mathcal{J}\mathcal{R}(U^{(n)}) &= \mathcal{R}(U^{(n)})\mathcal{L}_{n}\mathcal{J}, \label{3.5}  \\
	\mathcal{J}\mathcal{L}_{n}\mathcal{R}^{\star}(U^{(n)}) &= \mathcal{R}^{\star}(U^{(n)})\mathcal{J}\mathcal{L}_{n}. \label{3.6}
\end{align}
\end{coro}
Since $\mathcal{R}(Q)$ ($\mathcal{R}^*(Q)$) commutes with $L_m\mathcal{J}$ ($\mathcal{J}L_m$) by the factorizations (\ref{3.3}) and (\ref{3.4}), the operators $\mathcal{J}L_m$ and $\mathcal{R}^*(Q)$ share the same eigenfunctions, as do $L_m\mathcal{J}$ and $\mathcal{R}(Q)$. Their eigenvalues can be determined by analyzing the asymptotic behavior of these eigenfunctions.

We analyze the spectrum of $L_m$ as follows. First, we determine the spectrum of $\mathcal{J}L_m$, which is simpler than that of $L_m$ by relating it via (\ref{3.3}) to the spectrum of $\mathcal{R}^*(Q)$. We then show that the eigenfunctions of $\mathcal{R}^*(Q)$ (or $\mathcal{J}L_m$), together with a generalized kernel of $\mathcal{J}L_m$, form an orthogonal basis in $L^2(\mathbb{R})$ (a completeness relation). Finally, decomposing a function $z$ in this basis allows us to compute the quadratic form $\langle L_m z, z \rangle$ and directly obtain the inertia of $L_m$.
\subsection{The spectrum of the recursion operator around one soliton}
We firstly define the following squared eigenfunctions
\begin{align}
	F^+(x_1,x_2;k):=N(x_{1},k)\bar{N}^{*}(x_{2},k),\quad F^-(x_1,x_2;k):=\bar{N}^{*}(x_{1},k)N(x_{2},k),
\end{align}
and
\begin{align}
	F^{\pm}_n(x_1,x_2):=F^{\pm}(x_1,x_2,i\kappa_n),
\end{align}
where $\bar{N}^{*}(x_{2})$ and $\Phi_{1}^{*}(x_{2})$ satisfy the adjoint eigenvalue problem of (\ref{lax3}) for $u=Q$ (i.e., substituting $-i$ and $x_{2}$ for $i$ and $x$).

By employing the properties of the squared eigenfunctions associated with the eigenvalue problem (\ref{lax3}), the following result can be established.
\begin{lemma}\label{lem3.2}
	The recursion operator $\mathcal{R}(Q)$ defined in $L^{2}(\mathbb{R})$ with domain $H^{1}(\mathbb{R})$ possesses exactly one discrete eigenvalue $-c$ corresponding to the eigenfunction $Q$. Its essential spectrum is the interval $[0, +\infty)$, where the associated eigenfunctions lack spatial decay and are not in $L^{2}(\mathbb{R})$. Furthermore, the kernel of $\mathcal{R}$ is spanned by $F^+(x,0) =  (N\bar{N}^{*})(x,0)=|N(x,0)|^{2}$, with $N(x,\lambda)$ and $\bar{N}(x,\lambda)$ are Jost solutions to the spectral problem (\ref{lax3}).
\end{lemma}

\begin{proof}
	Consider Jost solutions for the spectral problem (\ref{lax3}) under potential $u=Q$, incorporating asymptotic expression from (\ref{2.13}). This configuration yields a single discrete eigenvalue $k=i\kappa_1$ ($0<\kappa_1<\frac{\delta}{2\pi}$), which generates the soliton profile $Q$. Then utilizing properties of the singular operator $T^\delta, T^\delta_{12}$ from Subsection \ref{sub2.1}, the spectral problem (\ref{lax3}) and the relation $Q_{12}^{-}Q_{12}^{+} = Q_{12}^{+}Q_{12}^{-}$
	derive for $k> 0$ that
	\begin{align}
		&\left(Q_{12}^{+} - iQ_{12}^{-}T^\delta_{12}\right)\left(Q_{12}^{-}F^\pm(x_1,x_2;k) \right) = 4(\lambda+\frac{1}{2\delta}) Q_{12}^{-}F^\pm(x_1,x_2;k), \label{3.8} \\
		&\left(Q_{12}^{+} - iQ_{12}^{-}T^\delta_{12}\right)\left(Q_{12}^{-}F^{+}_1(x_1,x_2)\right) = 4(\lambda+\frac{1}{2\delta}) Q_{12}^{-}F^{+}_1(x_1,x_2), \label{3.10}
	\end{align}
Then from (\ref{3.8})-(\ref{3.10}) we deduce
	\begin{align}
		&\mathcal{R}_{12}(Q)F^\pm(x_1,x_2;k)= -4(\lambda+\frac{1}{2\delta}) F^\pm(x_1,x_2;k), \\
		&\mathcal{R}_{12}(Q)F^{+}_1(x_1,x_2) = -4(\lambda_1 +\frac{1}{2\delta})F^{+}_1(x_1,x_2)
		= -2cF^{+}_1(x_1,x_2).
	\end{align}
	Accounting for the properties in the bi-Hamiltonian structure (\ref{1.14}), we could obtain
	\begin{align}
		&\mathcal{R}(Q)F^\pm(x,k) = -2(\lambda+\frac{1}{2\delta})F^\pm(x,k), \text{for } k > 0,\label{3.14}\\
		&\mathcal{R}(Q)F^{+}_1(x) = -2(\lambda+\frac{1}{2\delta})F^{+}_1(x)=-cF^{+}_1(x) . \label{3.16}
	\end{align}
	From $\frac{\delta H_2}{\delta u}=\mathcal{R}(Q)\frac{\delta H_1}{\delta u}$, we find  $\mathcal{R}(Q)Q = -cQ$.  Differentiating it with respect to $c$ yields
	
	\[
	\mathcal{R}(Q)\frac{\partial Q}{\partial c} = -Q - c\frac{\partial Q}{\partial c}.
	\]
	Equation (\ref{3.14}) implies that the essential spectrum of $\mathcal{R}(Q)$ is $(0, +\infty)$, corresponding to $-2(\lambda+\frac{1}{2\delta})> 0$. The generalized eigenfunctions $F^{\pm}(x,k)$ exhibit no spatial decay and are not in $L^{2}(\mathbb{R})$, as evidenced by (\ref{2.13}).
Similarly, the kernel of $R(Q)$ is attached at $k=0$ and the associated kernel is $F^+(x,0) = |N(x,0)|^{2} \notin L^{2}(\mathbb{R})$. This complete the proof.
\end{proof}
\begin{re}
Since $\mathcal{R}(Q)Q = -cQ$, we can infer from (\ref{3.16}) that $F^{+}_1(x)\sim Q(x)$.
\end{re}

Similar to the proof of Lemma \ref{lem3.2}, we have the following result concerning the spectrum of the composite operators $\mathcal{R}^n(Q)$ for $n\geq 2$.
\begin{coro}
	The  composite operator $\mathcal{R}^m(Q)$ defined in $L^{2}(\mathbb{R})$ with domain $H^{m}(\mathbb{R})$ has only one eigenvalue $(-c)^m$, the essential spectrum is the interval $(0, \infty)$, and the corresponding generalized eigenfunctions do not have spatial decay and not in $L^{2}(\mathbb{R})$.
\end{coro}

We now analyze the adjoint recursion operator $\mathcal{R}^{\star}(Q)$. Due to the factorization (\ref{3.4}), it shares eigenfunctions with $\mathcal{J}L_{n}$, making it particularly relevant for soliton spectral stability analysis. From (\ref{2.40}) we have
\begin{align*}
	\mathcal{R}^{\star}(u) &= J\mathcal{R}(u)J^{-1}.
\end{align*}
The spectral properties of $\mathcal{R}^{\star}(Q)$ are characterized as follows.
\begin{lemma}\label{lem3.3}
	The adjoint recursion operator $\mathcal{R}^{\star}(Q)$ defined in $L^{2}(\mathbb{R})$ with domain $H^{1}(\mathbb{R})$ possesses exactly one eigenvalue $-c$ with corresponding eigenfunction $Q_{x}$. Its essential spectrum is $(0, +\infty)$, where associated eigenfunctions lack spatial decay and do not belong to $L^{2}(\mathbb{R})$. Additionally, the kernel of $\mathcal{R}^{\star}(Q)$ is spanned by $(N\bar{N}^{*})_{x}(x,0)$.
\end{lemma}

\begin{proof}
	Consider Jost solutions for spectral problem (\ref{lax3}) with potential $u=Q$, incorporating asymptotic expressions from (\ref{2.13}). The soliton profile $Q$ originates from eigenvalue $k=i\kappa$ ($0<\kappa_1<\frac{\delta}{2\pi}$). Then following the approach in Lemma \ref{lem3.2}'s proof, we find eigenvalues of $\mathcal{R}^{\star}_{12}(Q)$ near $Q$ using equations (\ref{3.8})-(\ref{3.10}) and obtain
	\begin{align}
		&\mathcal{R}^{\star}_{12}(Q)(Q^-_{12}F^\pm(x_1,x_2;k))= -4(\lambda+\frac{1}{2\delta}) Q^-_{12}F^\pm(x_1,x_2;k), \\
		&\mathcal{R}^{\star}_{12}(Q)(Q^-_{12}F^{+}_1(x_1,x_2)) = -4(\lambda_1 +\frac{1}{2\delta})Q^-_{12}F^{+}_1(x_1,x_2)
		= -2cQ^-_{12}F^{+}_1(x_1,x_2).
	\end{align}
This leads to the following relations
	\begin{align}
		&\mathcal{R}^{\star}(Q)F^\pm_x(x,k) = -2(\lambda+\frac{1}{2\delta})F^\pm_x(x,k), \text{for } k > 0, \label{3.20}\\
		&\mathcal{R}^{\star}(Q)\partial_xF^{+}_1(x) = -2(\lambda+\frac{1}{2\delta})\partial_xF^{+}_1(x)=-c\partial_xF^{+}_1(x), \label{3.22}\\
		&\mathcal{R}^{\star}(Q) \frac{\partial Q_{x}}{\partial c} = -Q_{x} - c\frac{\partial Q_{x}}{\partial c}.
	\end{align}
Since $\mathcal{R}(Q)Q_x = -cQ_x$, by (\ref{3.22}), $-c$ is the only discrete eigenvalue.
 Equation (\ref{3.20}) shows that the essential spectrum of $\mathcal{R}(Q)$ is $(0, +\infty)$, corresponding to $-2(\lambda+\frac{1}{2\delta})> 0$. Generalized eigenfunctions $F_x^\pm (x,k)$ exhibit no spatial decay and are not in $L^{2}(\mathbb{R})$, as evidenced by (\ref{2.13}). Similarly, the kernel of $R^\star(Q)$ is attached at $k=0$ and the associated kernel is $F_x^+(x,0) =\partial_x|N(x,0)|^{2} \notin L^{2}(\mathbb{R})$. This complete the proof.
\end{proof}

\begin{re}
	Lemmas \ref{lem3.2} and \ref{3.3} demonstrate that $\mathcal{R}(Q)$ and $\mathcal{R}^{\star}(Q)$ are essentially invertible in $L^{2}(\mathbb{R})$.
\end{re}
\subsection{The spectrum of linearized operators $\mathcal{J}L_m, L_m\mathcal{J}$ and $L_m$}
We now consider the operator $\mathcal{J}L_m$. Since $L_m=\mathcal{R}^{m-1}(Q)L_1$ (defined in $L_2(\mathbb{R})$ with domain $H^n(\mathbb{R})$), the symbol of the principle part of the operator $\mathcal{J}L_n$ is
 \begin{align}
 	\varrho_{m,c}(\xi):=i\xi (2\pi \xi \coth(2\pi\delta \xi)-\frac{1}{\delta})^{m-1}(2\pi \xi \coth(2\pi\delta \xi)+c-\frac{1}{\delta}).\label{3.24}
 \end{align}
Then we have the following statement related to the spectra of the operator $\mathcal{J}L_m$.
\begin{prop}\label{prop3.1}
	The essential spectra of $\mathcal{J}L_m$ (defined in $L^{2}(\mathbb{R})$ with domain $H^{m+1}(\mathbb{R})$) for $m \geq 1$ is $i\mathbb{R}$, the kernel is spanned by the function $Q_{x}$ and the generalized kernel is spanned by $\frac{\partial Q}{\partial c}$.
\end{prop}

\begin{proof}
	The proof is by direct verification. We compute the spectrum of the operator $\mathcal{J}L_m$ directly by employing the squared eigenfunctions as follows
	\begin{align}
		&\mathcal{J}L_mF^\pm(x,k) = \varrho_{m,c}(\pm2k)F^\pm(x,k), \quad \text{for } k > 0,\label{3.25} \\
		&\mathcal{J}L_m\partial_xF^+_1(x) \sim \mathcal{J}L_m Q_{x} = 0, \label{3.27} \\
		&\mathcal{J}L_m \frac{\partial Q}{\partial c}= (-1)^{m} c^{m-1} Q_{x}. \label{3.28}
	\end{align}
	In view of (\ref{3.24}) and (\ref{3.25}), the essential spectrum of $\mathcal{J}L_m$ are $ \varrho_{m,c}(\pm2k)$ for $k > 0$, which is the whole imaginary axis. In view of (\ref{3.27}) and (\ref{3.28}), the kernel and generalized kernel of $\mathcal{J}L_m$ is $Q_{x}$ and $\frac{\partial Q}{\partial c}$, respectively. The proof is completed.
\end{proof}
For the adjoint operator of $\mathcal{J}L_m$, namely, the operator $-L_{m}\mathcal{J}$, for the spectrum of which, we have the following result.
\begin{prop}\label{prop3.2}
	The linear operator $L_{m}\mathcal{J}$, defined in $L^{2}(\mathbb{R})$ with domain $H^{m+1}(\mathbb{R})$ for $m \geq 1$, possesses an essential spectrum coinciding with the imaginary axis $i\mathbb{R}$. Its kernel is spanned by the soliton solution $Q$, while the generalized kernel is spanned by the integrated derivative $\partial_{x}^{-1}\left(\frac{\partial Q}{\partial c}\right)$.
\end{prop}

\begin{proof}
	Direct spectral analysis using squared eigenfunctions yields the following relations
	\begin{align}
		&L_{m}\mathcal{J}F^\pm(x,k) = \varrho_{m,c}(\pm2k)F^\pm(x,k), \quad \text{for}~ k > 0, \label{3.29}\\
		&L_{m}\mathcal{J}F^+_1(x) \sim L_{m}Q_{x}=0, \label{3.31} \\
		&L_{m}\mathcal{J}\partial_{x}^{-1}(\frac{\partial Q}{\partial c}) = (-1)^{m}c^{m-1}Q. \label{3.32}
	\end{align}
	
	Analysis of equations (\ref{3.24}) and (\ref{3.29}) shows the essential spectrum consists of points $\varrho_{m,c}(\pm2k)$ for $k > 0$, forming the entire imaginary axis. Equations (\ref{3.31}) and (\ref{3.32}) identify $Q$ and $\partial_{x}^{-1}\left(\frac{\partial Q}{\partial c}\right)$ as basis elements for the kernel and generalized kernel, respectively.
\end{proof}
 On account of Propositions \ref{prop3.1} and \ref{prop3.2}, we now have the two function sets as follows. The first set
\begin{align}
	 \left\{F^\pm_{x}(x,k), \quad \text{ for } k>0;\quad Q_{x};\quad
	\frac{\partial Q}{\partial c}\right\}\label{set1}
\end{align}
 comprises linearly independent eigenfunctions and generalized kernel elements of the operator $\mathcal{J}L_{m}$. These elements are essentially orthogonal under the $L^{2}$ inner product. The second set
\begin{align}
	\left\{F^\pm(x,k),\quad \text{ for } k>0;\quad Q;\quad \partial_{x}^{-1}\frac{\partial Q}{\partial c}\right\}\label{set2}
\end{align}
 consists of linearly independent eigenfunctions and generalized kernel elements of the operator $L_{m}\mathcal{J}$.

 Given the even symmetry of $\frac{\partial Q}{\partial c}$, for $k,k'>0$, we can employ asymptotic behaviors of Jost solutions (\ref{2.13}) to compute their inner products

 \begin{align}
 	&\int_{\mathbb{R}} F^\pm_{x}(x,k)F^\mp(x,k') dx = \mp 2\pi ik |a(k)|^2\delta(k-k'), \label{3.35}\\
 	&\int_{\mathbb{R}} F^\pm_{x}(x,\lambda)F^\pm(x,k') dx = 0, \\
 	&\int_{\mathbb{R}} Q_{x}\partial_{x}^{-1}\left(\frac{\partial Q}{\partial c}\right) dx= -\mathcal{G}^\delta_{\kappa,c}, \\
 	&\int_{\mathbb{R}} \frac{\partial Q}{\partial c} Q dx = \frac{\text{d}H_{1}(Q)}{\text{d}c} = \mathcal{G}^\delta_{\kappa,c}.
 \end{align}
 The closure relation demonstrates functional completeness
 \begin{align}\label{3.39}
 	\mp &\int_{0}^{\infty} \frac{1}{2\pi ik |a(k)|^2} F^\pm_{x}(x,k)F^\mp(y,k)\text{d}k \\\nonumber
 	&+ \frac{1}{\mathcal{G}^\delta_{\kappa,c}} \biggl[ Q(y)\frac{\partial Q(x)}{\partial c} - Q_{x}\partial_{y}^{-1}\frac{\partial Q(y)}{\partial c} \biggr] = \delta(x-y).
 \end{align}
 This implies any function $z(y)$ vanishing at infinity expands over the two bases (\ref{set1}) and (\ref{set2}). In particular,we have the following decomposition of the function $z$
 \begin{align}\label{decom}
 	z(x) = \int_{0}^{\infty} F^\pm_x(x,k) \alpha^\pm (k)\text{d}k + \beta Q_{x} + \gamma\frac{\partial Q}{\partial c},
 \end{align}
with coefficients $\alpha^\pm(k), \beta$ and $\gamma$ can be computed explicitly by the closure relation (\ref{3.39}). Similarly, the function $z(x)$ decomposes over the second set (\ref{set2}) by multiplying (\ref{3.39}) with $z(x)$ and integrating with d$x$.

With the decomposition of function $z(x)$ in (\ref{decom}), we can compute the quadratic form related to the operator $L_n$ and illustrate the spectral information. The following statement describes the full spectrum of linearized operator $L_m=H''_{m+1}(Q)+cH''_m(Q)$.
\begin{lemma}\label{lem3.4}
	For $m \geq 1$ and any $z \in H_{\mathrm{od}}^{\frac{m}{2}}(\mathbb{R})$, the quadratic form satisfies
	\[
	\langle L_{m}z,z\rangle \geq 0,
	\]
with equality holding if and only if $z$ is a constant multiple of $Q_{x}$.
Regarding the spectral structure in even symmetric spaces
	\begin{itemize}
\item[(i)] In $H_{\mathrm{ev}}^{\frac{m}{2}}(\mathbb{R})$ with odd $n$, the operator $L_{m}$ exhibits exactly one negative eigenvalue while having no eigenvalue at zero.
\item[(ii)] For even $n$ in $H_{\mathrm{ev}}^{\frac{m}{2}}(\mathbb{R})$, $L_{m}$ possesses no negative eigenvalues.
	\end{itemize}
\end{lemma}
\begin{proof}
For any $z(x) \in H^{\frac{m}{2}}(\mathbb{R})$, decompose $z(x)$ as in (\ref{decom}) and evaluate the quadratic form $\langle L_m z,z\rangle$
\begin{align}\label{3.41}
	\langle L_m z, z\rangle &= \langle \int_0^\infty \alpha^\pm(k) L_m F^\pm_{x}(x,k)  \text{d}k, \int_0^\infty \alpha^\pm(k) F^\pm_{x}(x,k) \text{d}k \rangle\\\nonumber
	&\quad+ 2\gamma \langle\int_0^\infty \alpha^\pm (k) L_m F^\pm_{x}(x,k)  \text{d}k,\frac{\partial Q}{\partial c}\rangle + \gamma^2 \langle L_m \frac{\partial Q}{\partial c},\frac{\partial Q}{\partial c}\rangle \\\nonumber
	&= I + II + III.
\end{align}		
From (\ref{3.29}) and orthogonality properties of the two sets (\ref{set1}) and (\ref{set2}), we can find that
\begin{align}\label{ii}
			II =2\gamma \int_0^\infty
			\langle F^\pm(x,k),\frac{\partial Q}{\partial c} \rangle \alpha^\pm(k) \varrho_{m,c}(\pm 2k) \text{d}k = 0.
\end{align}
For the third term, a direct computation shows that
\begin{align}\label{iii}
			III = \mathcal{G}^\delta_{\kappa,c} \gamma^2 (-1)^m c^{m-1}.
\end{align}
For term $I$, applying (\ref{3.29}) and (\ref{3.35}) yields
\begin{align}\label{i}
I &= \langle
\int_0^\infty \alpha^\pm (k) \varrho_{m,c}(\pm 2k) F^\pm(x,k) \text{d}k, \int_{0}^{\infty} \alpha^\pm (k) F^\pm_{x}(x,k) \text{d}k\rangle \geq 0,
\end{align}
which vanishes if and only $\alpha^\pm (k) = 0$. Combining (\ref{ii}), (\ref{iii}) and (\ref{i}), one has
\begin{align}\label{3.45}
\langle L_m z, z\rangle &=\mathcal{G}^\delta_{\kappa,c} \gamma^2 (-1)^m c^{m-1}+
4\pi\int_{0}^{\infty} \mathcal{F}_m^\delta (k)|a(k)|^2 |\alpha^\pm(k)|^2 \text{d}k,
\end{align}
with $\mathcal{F}_m^\delta(k) =k^2 (4\pi k\coth(4\pi\delta k)-\frac{1}{\delta})^{m-1}(4\pi k \coth(4\pi\delta k)+c-\frac{1}{\delta})$.

For $ z\in H_{od}^{\frac{m}{2}}(\mathbb{R}) $, we have $ \gamma=0 $, then (\ref{3.45}) and (\ref{i}) reveal that $ \langle L_{m}z,z\rangle\geq 0 $. Moreover, $ \langle L_{m}z,z\rangle=0 $ infers that $ \alpha^\pm(k)=0 $, therefore, $ z=\beta Q_{x} $ for $ \beta\neq0 $.

If $ z\in H_{ev}^{\frac{m}{2}}(\mathbb{R}) $, we then have $ \beta=0 $, In the hyperplane $ \gamma=0 $, $ \langle L_{m}z,z\rangle\geq 0 $ and $ \langle L_{m}z,z\rangle=0 $ if and only if $ \alpha^\pm(k)=0 $, then one has $ z=0 $. Therefore, $ \langle L_{m}z,z\rangle>0 $ in the hyperplane $ \gamma=0 $ and which implies that $ L_{m} $ can have at most one negative eigenvalue. If $ m $ is odd, then $ L_{m}\frac{\partial Q}{\partial c} = -c^{m-1}Q < 0 $ and $ \left\langle L_{m}\frac{\partial Q}{\partial c}, \frac{\partial Q}{\partial c} \right\rangle = \mathcal{G}^\delta_{\kappa,c}  (-1)^m c^{m-1} < 0 $. Therefore, $ L_{m} $ has exactly one negative eigenvalue. If $ m $ is even, then from (\ref{3.45}) or $ \left\langle L_{m}\frac{\partial Q}{\partial c}, \frac{\partial Q}{\partial c} \right\rangle = \mathcal{G}^\delta_{\kappa,c} c^{m-1} > 0 $, which means that $ L_{m} $ has no negative eigenvalue. This completes the proof of Lemma \ref{lem3.4}.
\end{proof}
\begin{re}
	Lemma \ref{lem3.4} states that for $ k \in \mathbb{N} $, the inertia of the operators $ L_{m} $ satisfy $ \mathrm{in}(L_{2k}) = (0,1) $ and $ \mathrm{in}(L_{2k-1}) = (1,1) $. One can verify, by Weyl's essential spectrum theorem, that the essential spectrum of $ L_{m} $ ($ m \geq 2 $) is the interval $ [0,+\infty) $. It is inferred from $ L_{2k} = \mathcal{R}(Q)L_{2k-1} $ or (\ref{3.45}) that the operator $ L_{2k} $ has a positive eigenvalue $ \nu = O(c^{2k}) $ (with $ L^{2} $-eigenfunctions), which may possibly be embedded into its continuous spectrum.
\end{re}

As a direct consequence of Lemma \ref{lem3.4}, one has the following spectral information of higher order linearized operators
\[
\mathcal{T}_{m,j} := H_{m+2}^{\prime\prime}(Q_{c_j}) + (c_1 + c_2)H_{m+1}^{\prime\prime}(Q_{c_j}) + c_1c_2H_m^{\prime\prime}(Q_{c_j}),
\]
(defined in $ L^{2}(\mathbb{R}) $ with domain $ H^{1}(\mathbb{R}) $) with $ n \geq 1 $, $ j=1,2 $ and $ c_1 \leq c_2 $, which are related closely to stability problem of the double solitons $ U^{(2)} $. Following the same line of the proof of Lemma \ref{lem3.4}, we have the following results.
\begin{coro}\label{coro3.3}
The operators $\mathcal{T}_{m,j}$ with $ m \geq 1 $, $ j=1,2 $ and $ c_1 \leq c_2 $ have the following properties
\begin{itemize}
	\item [(i)]	For $ m \geq 1 $ and $ c_1 = c_2 = c $, we have $ \mathcal{T}_{m,1} = \mathcal{T}_{m,2} \geq 0 $, and the eigenvalue zero is double with eigenfunctions $ Q_c' $ and $ \frac{\partial Q_c}{\partial c} $.
	\item [(ii)]For $ m \geq 1 $ odd and $ c_1 < c_2 $, the operator $ \mathcal{T}_{m,1} $ has one negative eigenvalue and $ \mathcal{T}_{m,2} \geq 0 $ is positive.
	\item [(iii)]For $ m \geq 1 $ even and $ c_1 < c_2 $, the operator $ \mathcal{T}_{m,1} $ is positive and $ \mathcal{T}_{m,2} \geq 0 $ has one negative eigenvalue.
	\item [(iv)]$ \mathcal{T}_{m,j} $ have zero as a simple eigenvalue with associated eigenfunctions $ Q_{c_j}' $.
\end{itemize}
\end{coro}

\begin{proof}
	Similar to the proof of Lemma \ref{lem3.4}, we study quadratic form related to the operator $ \mathcal{T}_{m,j} $ with $ z $ possessing the decomposition (\ref{decom}). One can verify that
	\begin{equation}\label{3.46}
		\mathcal{T}_{m,j}\frac{\partial Q_{c_j}}{\partial c_j} = (c_j - c_k)(-c_j)^{m-1}Q_{c_j}, \quad \text{for} \quad k \neq j \quad \text{and} \quad j,k=1,2.
	\end{equation}
	
	In particular, if $ c_1 = c_2 = c $, the function $ \frac{\partial Q_c}{\partial c} $ belongs to the kernel of $ \mathcal{T}_{m,1} $ and $ \mathcal{T}_{m,2} $. Notice that $ Q_c' $ always belongs to the kernel of which, therefore, zero eigenvalue is double with eigenfunctions $ Q_c' $ and $ \frac{\partial Q_c}{\partial c} $. The non-negativeness of $ \mathcal{T}_{m,1} $ and $ \mathcal{T}_{m,2} $ follow from the same argument of Lemma \ref{lem3.4}.
	
	If $ c_1 < c_2 $, then by (\ref{3.46}) and following the same line of the proof of Lemma \ref{lem3.4}, the operator $ \mathcal{T}_{2k+1,1} $ has a negative eigenvalue and $ \mathcal{T}_{2k+1,2} \geq 0 $, their zero eigenvalue are simple with associated eigenfunction $ Q_{c_j}' $; the operator $ \mathcal{T}_{2k,1} \geq 0 $ and $ \mathcal{T}_{2k,2} $ has a negative eigenvalue.
\end{proof}
\begin{re}
	Base on the facts above, the linearized operator $ \mathcal{L}_{2} $ defined in (\ref{VL}) around the double solitons profile $ U^{(2)} $ possesses the following property. The spectra $ \sigma(\mathcal{L}_{2}) $ trends to the union of $ \sigma(\mathcal{T}_{1,1}) $ and $ \sigma(\mathcal{T}_{1,2}) $ as $ t $ goes to infinity. Since from Corollary \ref{coro3.3}, we know the inertia
	$\mathrm{in}(\mathcal{T}_{1,1}) = (1,1)$ and $\mathrm{in}(\mathcal{T}_{1,2}) = (0,1)$, then it reveals that $\mathrm{in}(\mathcal{L}_{2}) = \mathrm{in}(\mathcal{T}_{1,1}) + \mathrm{in}(\mathcal{T}_{1,2}) = (1,2)$.
\end{re}
\subsection{The spectrum of linearized operator around $n$-solitons}
In order to prove Theorem \ref{thm1.1}, we need to know the spectral information of the operator $\mathcal{L}_n$ (\ref{VL}).
\begin{lemma}\label{lem3.5}
	The operator $\mathcal{L}_n$ with domain $H^{\frac{n}{2}}(\mathbb{R})$ verifies the following spectral property
	\begin{align}\label{3.47}
	\text{in}(\mathcal{L}_n)=(n(\mathcal{L}_n), z(\mathcal{L}_n))=\left([\frac{n+1}{2}],n\right).
	\end{align}
\end{lemma}

To this aim, for $ j=1,2,\ldots,n $, recall that $ L_{n,j}=S_{n}^{\prime\prime}(Q_{c_{j}}) $ is defined in (\ref{3.1}). The spectrum of $ \mathcal{L}_{n} $ tends to the unions of $ \sigma(L_{n,j}) $, that is
\[
\sigma(\mathcal{L}_{n}) \to \bigcup_{j=1}^{n}\sigma(L_{n,j}) \quad \text{as} \quad t\to+\infty.
\]
The result (\ref{3.47}) follows directly from the following statement which concerning the inertia of the operators $ L_{n,j} $, $ j=1,2,\cdots,n $.

\begin{prop}\label{prop3.3}
The operators $ L_{n,2k-1}$ and $ L_{n,2k} $ has following spectral properties.
\begin{itemize}
	\item[(i)] $ L_{n,2k-1} $ (defined in $ L^{2}(\mathbb{R}) $ with domain $ H^{n}(\mathbb{R}) $) has zero as a simple eigenvalue and exactly one negative eigenvalue for $ 1\leq k\leq\left[\frac{n+1}{2}\right] $, i.e,
	\[
	\mathrm{in}(L_{n,2k-1}) = (1,1);
	\]
	
	\item [(ii)]$ L_{n,2k} $ (defined in $ L^{2}(\mathbb{R}) $ with domain $ H^{n}(\mathbb{R}) $) has zero as a simple eigenvalue and no negative eigenvalues for $ 1\leq k\leq\left[\frac{n}{2}\right] $, i.e,
	\[
	\mathrm{in}(L_{n,2k}) = (0,1).
	\]
\end{itemize}
\end{prop}
\begin{proof}
	The proof follows the same line of the proof of Lemma \ref{lem3.4}. We consider the operator $ L_{n,j}= S_{n}^{\prime\prime}(Q_{c_{j}}) $ for $1\leq j\leq n$ and compute the quadratic form $\langle L_{n,j}z,z\rangle$ under a special decomposition of $z$ in (\ref{decom}). Recall from (\ref{3.1}) that the form of $L_{n,j}$ which is a combination of the operators $H_{m+1}^{\prime\prime}(Q_{c_{j}})+c_{j}H_{m}^{\prime\prime}(Q_{c_{j}}),$ and those $\sigma_{j,k}>0$ are the elementally symmetric functions of $c_{1},c_{2},\cdots,c_{j-1},c_{j+1},\cdots,c_{n}.$ Moreover, one has
	\begin{align}\label{3.48}
		L_{n,j}\frac{\partial Q_{c_{j}}}{\partial c_{j}} = -\prod_{k\neq j}^{n}(c_{k}-c_{j})Q_{c_{j}} := \Gamma_{j}Q_{c_{j}}.
	\end{align}
		The quadratic form $\langle L_{n,j}z,z\rangle$ (for $ z\in H^{\frac{n}{2}}(\mathbb{R}) $ ) can be evaluated similar to (\ref{3.41}) as follows
	\begin{align*}
		\langle L_{n,j}z,z\rangle
		= 4\pi\sum_{m=1}^{n} \int_{0}^{\infty} \mathcal{F}_m^\delta (k)|a(k)|^2 |\alpha^\pm(k)|^2 \text{d}k + \mathcal{G}^\delta_{\kappa,c}\gamma^{2}\Gamma_{j}.
	\end{align*}
One can check that the symbol of the principle part of $ L_{m,j} $ evaluated at $ 2k$ is
	\begin{align}
		S_{n}^{\prime\prime}(0)(2k) = \sum_{m=1}^{n}\sigma_{j,n-m}\varrho_{m,c_{j}}(2k) > 0.
	\end{align}
	Then the first term of the quadratic form $ \langle L_{n,j}z,z\rangle $ is nonnegative and equals to zero if and only if $ \alpha^\pm(k)=0. $
	
	If $j$ is even, then in view of the definition of $ \Gamma_{j} $ (\ref{3.48}), one has $ \Gamma_{j}>0 $ and $ \langle L_{n,j}z,z\rangle\geq 0 $ and $ \langle L_{n,j}z,z\rangle=0 $ if and only if $ \alpha^\pm(k)=0 $ and $ \gamma=0 $, which indicates that $ z=\beta Q^{\prime}_{c_{j}}. $ Hence $ L_{n,j}\geq 0 $ and zero is simple with associated eigenfunction $ Q^{\prime}_{c_{j}}. $
	
	If $j$ is odd, then one has $ \Gamma_{j}<0 $, we investigate $z$ in $ H_{ev}^{\frac{n}{2}}(\mathbb{R}) $ and $ H_{od}^{\frac{n}{2}}(\mathbb{R}) $, respectively.
	
	If $ z\in H_{od}^{\frac{n}{2}}(\mathbb{R}) $, then $ \gamma=0. $ Then one has $ \langle L_{n,j}z,z\rangle\geq 0 $ and $ \langle L_{n,j}z,z\rangle=0 $ if and only if $ \alpha^\pm(k)=0. $ Then $ z=\beta Q^{\prime}_{c_{j}} $ with $ \beta\neq 0 $, which indicates that zero is simple with associated eigenfunction $ Q^{\prime}_{c_{j}}. $
	
	If $ z\in H_{ev}^{\frac{n}{2}}(\mathbb{R}) $, then $ \beta=0. $ In the hyperplane $ \gamma=0 $, $ \langle L_{n,j}z,z\rangle\geq 0 $ and $ \langle L_{n,j}z,z\rangle=0 $ if and only if $ \alpha^\pm(k)=0. $ Therefore, $ \langle L_{n,j}z,z\rangle>0 $ in the hyperplane $ \gamma=0 $ and which implies that $ L_{n,j} $ can have at most one negative eigenvalue. Since
	\[
	L_{n,j}\frac{\partial Q_{c_{j}}}{\partial c_{j}} = \Gamma_{j}Q_{c_{j}} < 0
	\]
	and
\begin{align}\label{n=1}
	\Bigg\langle L_{n,j}\frac{\partial Q_{c_{j}}}{\partial c_{j}}, \frac{\partial Q_{c_{j}}}{\partial c_{j}} \Bigg\rangle = \Gamma_{j}\frac{dH_{1}(Q_{c_{j}})}{dc_{j}} < 0,
\end{align}
	therefore, $ L_{n,j} $ has exactly one negative eigenvalue. This implies the desired result as advertised in the statement of Proposition \ref{prop3.3}.
\end{proof}

\begin{proof}[Proof of Lemma \ref{lem3.5}]
	From the invariance of inertia of $ \mathcal{L}_{n} $, we know that
	\[
	\mathrm{in}(\mathcal{L}_{n}) = \sum_{j=1}^{n} \mathrm{in}(\mathcal{L}_{n,j}) = \left( \left\lfloor \frac{n+1}{2} \right\rfloor, n \right).
	\]
	The proof is concluded.
\end{proof}
\begin{re}
	In view of (\ref{3.5}) and (\ref{3.6}), one may also investigate the spectrum of the operator $ \mathcal{J}\mathcal{L}_{n} $ to show the spectral stability of the ILW $n$-solitons and then the spectrum of the operator $ \mathcal{L}_{n} $. The idea is conducted by employing the eigenvalue problem (\ref{lax3}), we can derive the eigenvalues and the associated eigenfunctions of the recursion operator around the $n$-solitons profile $ U^{(n)}(x) $. Then we need to show the eigenfunctions plus their derivatives with respect to the eigenvalues $ k_{j}=i\kappa_j $ ($j=1,2,\cdots,n$) form a basis in $ L^{2}(\mathbb{R}) $. Finally, by a direct verification of the quadratic form $ \langle\mathcal{L}_{n}z,z\rangle $ (with function $z$ decomposes upon the above bases), one can also derive the inertia of the operator $ \mathcal{L}_{n} $. In fact, we can show the following
	\begin{align*}
		n(\mathcal{L}_{n})= -\sum_{1 \leq j=2k-1 \leq n} \operatorname{sgn}\left( \left\langle \mathcal{L}_{n}\frac{\partial U^{(n)}}{\partial c_{j}},\frac{\partial U^{(n)}}{\partial c_{j}} \right\rangle \right)
		= \left[\frac{n+1}{2}\right],
	\end{align*}
 for $k=1,2,\cdots,\left[\frac{n+1}{2}\right]$. This formula reveals that the negative eigenvalues of $ \mathcal{L}_{n} $ are generated by the directions $ \frac{\partial U^{(n)}}{\partial c_{j}} $ for odd $ j=1,3,\cdots,2\left[\frac{n+1}{2}\right]-1 $.

\end{re}
\section{Proof of the main results}\label{sec.4}
The $n$-soliton solution $U^{(n)}(t,x)$ does not minimize $\mathcal{S}_n$ globally but constitutes a constrained, non-isolated critical point for the optimization problem
\[
\min H_{n+1}(u(t)) \quad \text{subject to} \quad H_j(u(t)) = H_j(U^{(n)}(t)), \quad j = 1,2,\dots,n.
\]

Consider the self-adjoint operator $\mathcal{L}_n(t)$ defined by (\ref{1.22}), whose second variation properties are essential. Denote by $n(\mathcal{L}_n(t))$ the count of its negative eigenvalues. These quantities exhibit potential time-dependence. The $n \times n$ Hessian matrix is defined as
\begin{align}\label{4.1}
	D(t) := \left\{ \frac{\partial^2 \mathcal{S}_n(U^{(n)}(t))}{\partial \mu_i \partial \mu_j} \right\},
\end{align}
with $p(D(t))$ representing its positive eigenvalue count. Since $\mathcal{S}_n(t)$ remains invariant under ILW dynamics, $D(t)$ is time-independent.
\subsection{Dynamical stability of $n$-soliton solutions: Proof of theorem \ref{thm1.1}}
The proof of Theorem \ref{thm1.1} depends on the following fundamental result originally formulated by Maddocks and Sachs \cite{MS}. For convenience, we give a detailed proof here.
\begin{prop}\label{prop4.1}
	Assuming the spectral condition
\begin{align}\label{4.2}
		n(\mathcal{L}_n) = p(D)
\end{align}
	holds, then for some $C > 0$, the multi-soliton $U^{(n)}$ constitutes a non-degenerate unconstrained minimum of the augmented Lagrangian
\begin{align}\label{4.3}
		\Delta(u) := \mathcal{S}_n(u) + \frac{C}{2} \sum_{j=1}^{n} \left( H_j(u) - H_j(U^{(n)}) \right)^2.
\end{align}
As a consequence, $U^{(n)}(t,x)$ is dynamically stable.
\end{prop}

Therefore, to complete the proof of Theorem \ref{thm1.1}, it is sufficient to verify (\ref{4.2}). We start with the count of the number of positive eigenvalues of the Hessian matrix $D$, the results are as follows.
\begin{lemma}\label{lem4.1}
	For all $\mathbf{c} = (c_1, \ldots, c_n)$, $\mathbf{x} = (x_1, \ldots, x_n)$ with $0 < c_1 < \cdots < c_n$, we have
	\[
	p(D) = \left\lfloor \frac{n+1}{2} \right\rfloor.
	\]
\end{lemma}
\begin{proof}
	Let $t$ be fixed. For the convenience of notation, we will omit the dependence on $t$ in the proof since the results are independent of $t$ in any case. For all $1\leq i,j\leq n$, we have
	\begin{align}
		D_{ij}=\frac{\partial^2 \mathcal{S}_n}{\partial\mu_i \partial \mu_j}=\sum_{k=1}^{n}\frac{\partial c_k}{\partial \mu_i}\frac{\partial H_j}{\partial c_k},
	\end{align}
Here we utilize such a fact that
\begin{align}
	\frac{\partial \mathcal{S}_n}{\partial \mu_j}=H_j(U^{(n)}).
\end{align}
We can observe that $D$ can be expressed as the product of two matrices, namely,
\begin{align}
	D=AB,\quad A=(\frac{\partial c_j}{\partial \mu_i}), \quad B=(\frac{\partial H_j}{\partial c_i}).
\end{align}
For each $Q_{c_{j}}$ that constitutes the asymptotic form of the multi-solitons $U^{(n)}$, the value of $\frac{\partial H_j}{\partial c_i}$ is clearly known (see (\ref{2.30})), therefore we have
\begin{align}
	\frac{\partial H_j}{\partial c_i}=(-1)^{j-1}\frac{\mathcal{G}^\delta_{\kappa_j,c_j}}{c_i} {c_i}^j.
\end{align}
We can easily find that the matrix $B$ is invertable and $ B^{T}MB = B^{T}A $, by Sylvester's law of inertia, to find the number of positive eigenvalues of $D$ it suffices to consider the number of positive eigenvalues of the matrix $ (-1)^{n}B^{T}A $, we next evaluate this matrix product explicitly and observe that the answer is a diagonal matrix with entries of alternating sign, which facts follow immediately from the binomial expansion. The diagonal $(j, j)$ entry is
\begin{align}
	(-1)^{n+1}\frac{\mathcal{G}^\delta_{\kappa_j,c_j}}{c_j}\prod_{k\neq j}(c_{k}-c_{j}),
\end{align}
with all off-diagonal entries zero. With the assumed ordering of the speeds $ c_{j} $ these diagonal entries are of alternating sign, with $ \left\lfloor \frac{n+1}{2} \right\rfloor $ positive entries.
\end{proof}
\begin{proof}[\textbf{Proof of Theorem \ref{thm1.1}}]
	By Lemma \ref{lem4.1} and Lemma \ref{lem3.5}, one has that
	\[
	n(\mathcal{L}_{n}) = p(D) = \left\lfloor \frac{n+1}{2} \right\rfloor.
	\]
	The proof of Theorem \ref{thm1.1} is obtained directly in view of Proposition \ref{prop4.1}, since $U^{(n)}(t,x)$ is now an (non-isolated) unconstrained minimizer of the augmented Lagrangian (\ref{4.3}) which therefore serves as a Lyapunov function.
\end{proof}
\subsection{Orbital stability of double solitons: Proof of theorem \ref{thm1.2}}
In this section, our attention is now turned to the proof of Theorem \ref{thm1.2}. We need to prove a coercivity property on the Hessian of action related to the double-solitons profile $U$
\[
\mathcal{S}_{2}(U) = H_{3}(U) + (c_{1} + c_{2}) H_{2}(U) + c_{1}c_{2}H_{1}(U).
\]

As showed in Corollary \ref{coro3.3}, the linearized operator $ \mathcal{L}_{2} $ ((\ref{1.22}) with $ n=2 $) around the ILW double soliton possesses only one negative eigenvalue and the inertia of $ \mathcal{L}_{2} $ equals to $(1, 2)$. It is natural to verify that the kernel of $ \mathcal{L}_{2} $ is spanned by functions $ U_{(j)}:=\partial_{x_{j}}U $, where $ j=1,2 $ and $ x_{j} $ the spatial transitions. Let $ U_{-1} $ be an eigenfunction associated to the unique negative eigenvalue of the operator $ \mathcal{L}_{2} $, as stated in Corollary \ref{coro3.3}. We assume that $ \|U_{-1}\|_{L^{2}}=1 $. Then $ U_{-1} $ is unique and one has $ \mathcal{L}_{2}U_{-1}=-\lambda_{0}^{2}U_{-1} $ with $ -\lambda_{0}^{2} $ the associated negative eigenvalue. It is now easy to see the following result holds.
\begin{lemma}\label{lem4.2}
	Let $U$ be the ILW double-solitons with wave velocities $0<c_1<c_2$, and $U_{(1)},U_{(2)}$ be in the corresponding kernel of the operator $\mathcal{L}_2$. There exists $v_1>0$ depending on $c_1$ and $c_2$ only, such that for any $z\in H^1(\mathbb{R})$ satisfying the following orthogonality conditions
\begin{align}
		(z,U_{-1})_{L^2} = (z,U_{(1)})_{L^2} = (z,U_{(2)})_{L^2} = 0,
\end{align}
	then we have $\langle\mathcal{L}_2 z,z\rangle \geq v_1\|z\|_{H^1}^2$.
\end{lemma}

It is observed that $U_{-1}$ is hard to handle in our case, so we need a more applicable version of Lemma \ref{lem4.1}. To see this, we consider the natural modes associated to the scaling parameters, which are the best candidates to generate negative directions for the related quadratic form defined from $\mathcal{L}_2$. More precisely, for $t$ fixed, $i,j=1,2$ and $i\neq j$
\begin{align}
	\mathcal{L}_2\partial_{c_j}U = -(H_2'(U)+c_iH_1'(U)).
\end{align}
We now define a function $\Psi := \frac{\partial_{c_1}U-\partial_{c_2}U}{c_1-c_2}$. It is then found that $\Psi$ is Schwartz and satisfies
\begin{align}\label{4.11}
	\mathcal{L}_2\Psi = -U,
\end{align}
and
\begin{align}\label{4.12}
 \langle\mathcal{L}_2\Psi,\Psi\rangle = \frac{1}{c_2-c_1}(\mathcal{G}^\delta(\kappa_1)-\mathcal{G}^\delta(\kappa_2)) < 0,
\end{align}
in view of (\ref{2.30}) and $0<\kappa_1<\kappa_2<\frac{\pi}{2\delta}$ when $0<c_1<c_2$.
Then we have the following result which gives a coercivity property of $\langle\mathcal{L}_2 z,z\rangle$.
\begin{lemma}\label{lem4.3}
	Let $U$ be the ILW double solitons with wave velocities $0<c_{1}<c_{2}$, and $U_{(1)}$ and $U_{(2)}$ be in the corresponding kernel of the associated linearized operator $\mathcal{L}_{2}$. There exists $v_{2}>0$ depending only on $c_{1},c_{2}$, such that for any $z\in H^{1}(\mathbb{R})$ satisfying the following orthogonality conditions
\begin{align}
	(z, U_{(1)})_{L^{2}}=(z, U_{(2)})_{L^{2}}=0,
\end{align}
	then we have
\begin{align}
	\langle\mathcal{L}_{2}z,z\rangle\geq v_{2}\|z\|_{H^{1}}^{2}-\frac{1}{\mu_{1}}(z,U)_{L^{2}}^{2}.
\end{align}
\end{lemma}
\begin{proof}
	It suffices to show that under the conditions above and the orthogonality condition $(z,U)_{L^{2}}=0$, there holds
\begin{align}
	\langle\mathcal{L}_{2}z,z\rangle\geq v_{2}\|z\|_{H^{1}}^{2}.
\end{align}
	First notice that by (\ref{4.11}) and (\ref{4.12}) we have
\begin{align}
		(\Psi,U)_{L^{2}}=-\langle\mathcal{L}_{2}\Psi,\Psi\rangle>0.
\end{align}
The next step is to decompose $z$ and $\Psi$ in $\text{span}\{U_{-1},U_{(1)},U_{(2)}\}$ and the associated orthogonal subspace. We decompose
\begin{align}
	z &= \tilde{z} + p U_{-1}, \quad \Psi = \Psi_0 + n U_{-1} + a U_{(1)} + b U_{(2)}, \quad p,n,a,b \in \mathbb{R},
\end{align}
where the following orthogonal conditions hold
\begin{align}
	(\tilde{z}, U_{-1})_{L^2} &= (\tilde{z}, U_{(1)})_{L^2} = (\tilde{z}, U_{(2)})_{L^2} = 0, \\\nonumber
	(\Psi_0, U_{-1})_{L^2} &= (\Psi_0, U_{(1)})_{L^2} = (\Psi_0, U_{(2)})_{L^2} = 0.
\end{align}
And in addition,
\begin{align}
	(U_{-1}, U_{(1)})_{L^2} &= (U_{-1}, U_{(2)})_{L^2} = 0.
\end{align}
From the above identities, we obtain
\begin{align}\label{4.20}
	\langle \mathcal{L}_2 z, z \rangle = \langle \mathcal{L}_2 \tilde{z}, \tilde{z} \rangle - p^2 \lambda_0^2.
\end{align}
In view of (\ref{4.11}) and the self-adjointness of $\mathcal{L}_2$, we deduce that
\begin{align}\label{4.21}
	0 = (z, U)_{L^2} = -\langle \mathcal{L}_2 \Psi, z \rangle = -\langle \Psi, \mathcal{L}_2 z \rangle = -\langle \Psi_0, \mathcal{L}_2 \tilde{z} \rangle + p n \lambda_0^2.
\end{align}
On the other hand,
\begin{align}\label{4.22}
	(\Psi, U)_{L^2}= -\langle \mathcal{L}_2 \Psi_0, \Psi_0 \rangle + \lambda_0^2 n^2.
\end{align}
Combining (\ref{4.20}), (\ref{4.21}) and (\ref{4.22}), it follows from the Cauchy-Schwartz inequality that
\begin{align}
	\langle \mathcal{L}_2 z, z \rangle = \langle \mathcal{L}_2 \tilde{z}, \tilde{z} \rangle - \frac{\langle \mathcal{L}_2 \tilde{z}, \Psi_0 \rangle^2}{\langle \mathcal{L}_2 \Psi_0, \Psi_0 \rangle + (\Psi, U)_{L^2}}
	\geq \gamma_1 \langle \mathcal{L}_2 \tilde{z}, \tilde{z} \rangle,
\end{align}
where $0 < \gamma_1 < 1$.
From (\ref{4.20}), it is inferred that $\langle \mathcal{L}_2 \tilde{z}, \tilde{z} \rangle \geq p^2 \lambda_0^2 \geq 0$. This in turn implies that there exists a constant $C > 0$ such that
\begin{align}
	\langle \mathcal{L}_2 z, z \rangle \geq C \| z \|_{H^1}^2,
\end{align}
thereby concluding the proof of Lemma \ref{lem4.3}.
\end{proof}
Next, we are going to give the proof of Theorem \ref{thm1.2}. To this end, we employ a natural Lyapunov functional $ \mathcal{S}_{2} $ for the ILW equation, which is well-defined at the natural $ H^{1} $ level. Indeed, for any $ u_{0}\in H^{1}(\mathbb{R}) $, we have global in time $ H^{1}(\mathbb{R}) $ solution $ u(t) $ \cite{IS}. Recall the Lyapunov functional $ \mathcal{S}_{2} $ defined by
\begin{align}
	\mathcal{S}_{2}(u(t)) = H_{3}(u(t)) +(c_{1}+c_{2})H_{2}(u(t)) + c_{1}c_{2}H_{1}(u(t))
	= \mathcal{S}_{2}(u_{0}).
\end{align}
It is clear that $ I_{2}(u) $ represents a real-valued conserved quantity, well-defined for $ H^{1} $-solutions of the ILW equation. Moreover, one has the following Taylor-like expansion.
\begin{lemma}\label{lem4.4}
	Let $z \in H^{1}(\mathbb{R})$ be any function with sufficiently small $H^{1}$-norm, and $U$ be the 2-solitons. Then, for all $t \in \mathbb{R}$, one has
	\begin{align}\label{4.26}
		\mathcal{S}_{2}(U + z) = I_{2}(U) + \frac{1}{2} \langle \mathcal{L}_{2} z, z \rangle + \mathcal{N}(z), \quad
		|\mathcal{N}(z)| = O(\|z\|_{H^{1}}^{3}),
	\end{align}
	where $\mathcal{L}_{2}$ is the linearized operator around $U$ and $\mathcal{N}(z)$ denotes the nonlinear remainder term.
\end{lemma}
\begin{proof}
	This is a direct consequence of the fact that $U$ is a critical point of the functional $\mathcal{S}_{2}$, which implies that the first variation of $\mathcal{S}_{2}$ at $U$ vanishes.
\end{proof}

The Lyapunov functional $\mathcal{S}_{2}$ plays a crucial role in describing the dynamics of small perturbations around the double soliton solution. The degenerate directions are
controlled by $H^1$ conservation law $H^1$.

\begin{proof}[\textbf{Proof of Theorem \ref{thm1.2}}]
	Let $u_0 \in H^1(\mathbb{R})$ be a function such that $u_0$ satisfying (\ref{u02}). Assume $u(t)$ is the solution of the Cauchy problem associated to the ILW equation with initial data $u_0$. In view of the continuity of the ILW flow for $H^1$ data in proposition \ref{prop1}, there exists a time $T_0 > 0$ and continuous parameters $x_1(t), x_2(t) \in \mathbb{R}$, defined for all $t \in [0, T_0]$, and such that the solution $u(t)$ of the Cauchy problem associated with the ILW equation with initial data $u_0$, satisfies
	\begin{align}\label{4.27}
		\sup_{t \in [0, T_0]} \| u(t) - U(t; x_1(t), x_2(t)) \|_{H^1(\mathbb{R})} < 2\epsilon.
	\end{align}
	We want to show that $T_0 = +\infty$. To this aim, let $K > 2$ be a constant, to be fixed later. Let us suppose, by contradiction, that the maximal time of stability $T^{\star}$, that is
	\begin{align}\label{4.28}
		T^{\star} = \sup \Big\{ T > 0, \text{ for all } t \in [0, T], \text{ there exists } \tilde{x}_1(t), \tilde{x}_2(t) \in \mathbb{R}, \nonumber \\
		\text{ such that }
		\sup_{t \in [0, T]} \| u(t) - U(t; \tilde{x}_1(t), \tilde{x}_2(t)) \|_{H^1(\mathbb{R})} < K\epsilon \Big\},
	\end{align}
	is finite. By (\ref{4.27}), we see easily that $T^{\star}$ is well-defined. Our idea is to find a suitable contradiction to the assumption $T^{\star} < +\infty$.
	
	By taking $\epsilon_0$ smaller, if necessary, we can apply modulation theory for the solution $u(t)$. We now give the following claim.
\begin{prop}
	 There exists $\epsilon_0 > 0$, depending only on $U$, such that for all $\epsilon \in (0, \epsilon_0)$, the following property is verified. There exist $x_1(t), x_2(t) \in \mathbb{R}$ defined on $[0, T^{*}]$, such that if we denote
	\begin{align}\label{4.29}
		\Upsilon(t) = u(t) - U(t), \quad U(t) = U(t; x_1(t), x_2(t)),
	\end{align}
	then for all $t \in [0, T^{*}]$, $\Upsilon$ satisfies the orthogonality conditions
	\begin{align}\label{4.30}
		(\Upsilon, U_{(1)})_{L^2} = (\Upsilon, U_{(2)})_{L^2} = 0.
	\end{align}
	Moreover, for all $t \in [0, T^{*}]$, we have
	\begin{align}\label{4.31}
		\|\Upsilon(t)\|_{H^1} + |x_1^{\prime}(t)| + |x_2^{\prime}(t)| < CK\epsilon, \quad \|\Upsilon(0)\|_{H^1} \leq C\epsilon,
	\end{align}
	for some constant $C > 0$, independent of $K$.
\end{prop}
	
	The proof of this Proposition relies on the Implicit Function Theorem. Indeed, let
	\begin{align}
		J_j(u(t), x_1, x_2) = \langle u - U(t; x_1, x_2), U_{(j)}(t; x_1, x_2) \rangle.
	\end{align}
	We clearly have
	\begin{align}
		J_j(U(t; x_1, x_2), x_1, x_2) = 0.
	\end{align}
	On the other hand, for $j, k = 1, 2$, we have
	\begin{align}
		\partial_{x_k} J_j(u(t, x), x_1, x_2) \big|_{(U, 0, 0)} = -\langle U_{(k)}(t; 0, 0), U_{(j)}(t; 0, 0) \rangle.
	\end{align}
	Let $J$ be the $2 \times 2$ matrix with component $J_{j,k} := (\partial_{x_k} J_j)_{j,k=1,2}$. We now compute
	\begin{align}
		\det J = \| U_{(1)} U_{(2)} \|_{L^2}^2 - \| U_{(1)} \|_{L^2}^2 \| U_{(2)} \|_{L^2}^2 \neq 0,
	\end{align}
	since the fact that $U_{(1)}$ and $U_{(2)}$ are not parallel for all time. Therefore, we have the desired invertibility, by the Implicit Function Theorem, we can write the decomposition (\ref{4.29}) with property (\ref{4.30}) in a small $H^1$ neighborhood of $U$, for $t \in [0, T^{\star}]$.
	
	Now we verify (\ref{4.31}). The first bounds are consequence of the decomposition itself and the equations satisfied by the derivatives of the parameters $x_1$ and $x_2$, after taking time derivative in (\ref{4.30}) and using the invertibility property of $\nabla J$. More precisely, we first write the equation verified by $\Upsilon$. Recall that $u$ satisfies the ILW equation, then one replaces $u$ by $U(t) + \Upsilon(t)$ in the ILW equation to obtain
	\begin{align}\label{4.36}
			\Upsilon_t + U_{(1)} x_1^{\prime}(t) + U_{(2)} x_2^{\prime}(t) +\frac{1}{\delta} \Upsilon_x+T^\delta \Upsilon_{xx} +\mathcal{N}(\Upsilon) = 0,
	\end{align}
	where $\mathcal{N}(\Upsilon)$ is the remaining nonlinear part.
	
	Take now the scalar product of (\ref{4.36}) with $U_{(j)}$. By the definition of $U$ and the orthogonality conditions (\ref{4.30}), we have
	\begin{align}
		\begin{split}
			&A \cdot (x_1^{\prime}(t), x_2^{\prime}(t))^T = B(\Upsilon) + O(\| \Upsilon \|_{H^1}^2), \\
			&A = \begin{pmatrix}
				\| U_{(1)} \|_{L^2}^2 & 0 \\
				0 & \| U_{(2)} \|_{L^2}^2
			\end{pmatrix} + \text{small},
		\end{split}
	\end{align}
	and $|B(\Upsilon)| \leq C \| \Upsilon \|_{H^1}$. As long as the modulation parameter do not vary too much and $\| \Upsilon \|_{H^1}$ remains small, $A$ is invertible and we can deduce that
	\begin{align}
		|x_1^{\prime}(t)| + |x_2^{\prime}(t)| \leq C \| \Upsilon \|_{H^1} + O(\| \Upsilon \|_{H^1}^2).
	\end{align}
	This completes the proof of the claim.
	
	Next, since $\Upsilon(t)$ defined by (\ref{4.29}) is small, by Lemma \ref{lem4.4} and the claim above, it is deduced that
	\begin{align}
		\mathcal{S}_2(u(t)) = \mathcal{S}_2(U(t)) + \frac{1}{2} \langle \mathcal{L}_2 \Upsilon(t), \Upsilon(t) \rangle + O(\| \Upsilon(t) \|_{H^1}^3).
	\end{align}
	By Lemma \ref{lem4.4}, it follows from (\ref{4.26}) that
	\begin{align}
		\begin{split}
			\langle \mathcal{L}_2 \Upsilon(t), \Upsilon(t) \rangle\leq C \| \Upsilon(0) \|_{H^1}^2 + C \| \Upsilon(t) \|_{H^1}^3.
		\end{split}
	\end{align}
	Since $\Upsilon(t)$ satisfies (\ref{4.31}), it is thus inferred from Lemma \ref{lem4.3} that
	\begin{align}\label{4.41}
		\begin{split}
			\| \Upsilon(t) \|_{H^1}^2 \leq C \epsilon^2 + C K^3 \epsilon^3 + C (U(t), \Upsilon(t))_{L^2}^2.
		\end{split}
	\end{align}
	Using the conservation of mass, it is found that
	\begin{align}
		\begin{split}
			\| u(t) \|_{L^2}^2  = \| u(0) \|_{L^2}^2.
		\end{split}
	\end{align}
	For $t \in [0, T^{\star}]$, it thus transpires that
	\begin{align}
		| (U(t), \Upsilon(t))_{L^2} | \leq C (\epsilon + K^2 \epsilon^2).
	\end{align}
	Replacing this last identity in (\ref{4.41}), by choosing $K$ large, one has
	\begin{align}
		\| \Upsilon(t) \|_{H^1}^2 \leq \frac{1}{2} K^2 \epsilon^2,
	\end{align}
	that is, $\| \Upsilon(t) \|_{H^1} \leq \frac{\sqrt{2}}{2} K \epsilon$. However, this estimate contradicts the definition of $T^{\star}$ in (\ref{4.28}) and therefore the stability of $U$ in $H^1$ is established, which completes the proof of Theorem \ref{thm1.2}.
\end{proof}

Next, we will proceed to prove Theorem \ref{th1.3}.
\subsection{The negative eigenvalues of $\mathcal{L}_n$ :Proof of Theorem \ref{th1.3}}
In this subsection, our goal is to prove the Theorem \ref{th1.3}.
\begin{proof}[\textbf{Proof of Theorem \ref{th1.3}}]
	The linearized operators around the $n$-solitons $\mathcal{L}_{n} = S_{n}^{\prime\prime}(U^{(n)})$ possess $\left[\frac{n+1}{2}\right]$ negative eigenvalues, as previously verified in (\ref{3.47}). We now prove the expression (\ref{negetive}) for these eigenvalues. As $t\to \infty$, the spectrum $\sigma(\mathcal{L}_{n}(t))$ converges to the union of the spectra $\sigma(\mathcal{L}_{n,j})$ where $\mathcal{L}_{n,j} = S_{n}^{\prime\prime}(Q_{c_{j}})$, that is
	\begin{align}
		\sigma(\mathcal{L}_{n}(t)) \rightarrow \bigcup_{j=1}^{n} \sigma(\mathcal{L}_{n,j}), \quad \text{as} \quad t \rightarrow +\infty.
	\end{align}
	Since the operators $\mathcal{L}_{n}(t)$ are isoinertial for each $n$, the spectrum $\sigma(\mathcal{L}_{n}(t))$ is independent of $t$. Thus, the negative eigenvalues of $\mathcal{L}_{n}$ match those of $\mathcal{L}_{n,j}$ for all $j = 1, 2, \dots, n$. According to Lemma \ref{lem3.5}, $\mathcal{L}_{n,j}$ has negative eigenvalues when $j = 2k-1$ and $1 \leq k \leq \left[\frac{n+1}{2}\right]$. Specially, when $n=1$, one has $\mathcal{L}_1=L_1$. Therefore, (\ref{n=1}) indicates that $\mathcal{L}_1$ has the negative $\lambda_1<0$. With the methods in \cite{BBSSB}, we could compute $\lambda_1$, which is
	\begin{align}\label{1.30}
		\lambda_1=-\frac{c-\frac{1}{\delta}}{2a^2\sin^2(a\delta)}<0,
	\end{align}
	Then for $n=2$, we have
	\begin{align}\label{lamb}
		\mathcal{L}_{2,1}=(\mathcal{R}(Q_{c_1})+c_2)\mathcal{S}_1''(Q_{c_1}).
	\end{align}
In terms of Lemma \ref{lem3.2}, the operator $(\mathcal{R}(Q_{c_1}) + c_{2})$ has an eigenvalue $c_{2} - c_{1} > 0$, with continuous spectrum $(c_2, +\infty)$ whose generalized eigenfunctions are not in $L^{2}(\mathbb{R})$. Therefore, the unique negative eigenvalues of $\mathcal{L}_{2,1}$ is $(c_2-c_1)\lambda_1<0$.\\
Assuming the result holds for $n = p$, namely, the $\left[\frac{p+1}{2}\right]$-th negative eigenvalue of $\mathcal{L}_{p}$ is
	\begin{align}\label{4.48}
		\lambda_{k}^{(p)} = -M^\delta (c_{2k-1}-\frac{1}{\delta}) \prod_{j \neq 2k-1}^{p} (c_{j} - c_{2k-1}), \quad k = 1, 2, \dots, \left[\frac{p+1}{2}\right].
	\end{align}
Then for $n = p+1$ is even, we have $\left[\frac{p+1}{2}\right] = \left[\frac{p+2}{2}\right]$ and for $k = 1, 2, \dots, \left[\frac{p+1}{2}\right]$, we have
	\begin{align}
		\mathcal{L}_{p+1,2k-1} =(\mathcal{R}(Q_{c_{2k-1}}) + c_{p+1}) \mathcal{S}_{p}^{\prime\prime}(Q_{c_{2k-1}}).
	\end{align}
	By Lemma \ref{lem3.2}, the operator $(\mathcal{R}(Q_{c_{2k-1}}) + c_{p+1})$ has an eigenvalue $c_{p+1} - c_{2k-1} > 0$, with continuous spectrum $(c_{p+1}, +\infty)$ whose generalized eigenfunctions are not in $L^{2}(\mathbb{R})$. Therefore, the $\left[\frac{p+2}{2}\right]$-th negative eigenvalues of $\mathcal{L}_{p+1,2k-1}$ is
	\begin{align}\label{4.50}
		\lambda_{k}^{(p+1)}= -M^\delta (c_{2k-1}-\frac{1}{\delta}) \prod_{j \neq 2k-1}^{p+1} (c_{j} - c_{2k-1}), \quad k = 1, 2, \dots, \left[\frac{p+1}{2}\right].
	\end{align}
For the case where $n = p+1$ is odd, we have $\left[\frac{p+1}{2}\right] + 1 = \left[\frac{p+2}{2}\right]$. The first $\left[\frac{p+1}{2}\right]$ negative eigenvalues of $\mathcal{L}_{p+1}$ are given by expression (\ref{4.50}). We now compute the final negative eigenvalue. Consider
\begin{align}
	\mathcal{L}_{p+1,p+1} =\left( \mathcal{R}(Q_{c_{p+1}}) + c_{j} \right) \tilde{S}_{p}^{\prime\prime}(Q_{c_{p+1}})
\end{align}
where $\tilde{S}_{p}$ denotes the action with wave speed $c_j$ in $S_p$ replaced by $c_{p+1}$ for some $1 \leq j \leq p$. By assumption (\ref{4.48}), the discrete eigenvalue of $\tilde{S}_{p}^{\prime\prime}(Q_{c_{p+1}})$ is
\begin{align}
	-M^\delta (c_{p+1}-\frac{1}{\delta})\prod_{l \neq j}^{p} (c_l - c_{p+1}).
\end{align}
The operator $\mathcal{R}(Q_{c_{p+1}}) + c_j$ has eigenvalue $c_j - c_{p+1} < 0$ by Lemma \ref{lem3.2}, with continuous spectrum $(c_j, +\infty)$ whose generalized eigenfunctions are not in $L^2(\mathbb{R})$. Therefore, the last negative eigenvalue of $\mathcal{L}_{p+1}$ is
\begin{align*}
	\lambda_{\left[\frac{p+2}{2}\right]}^{(p+1)}  = -M^\delta (c_{p+1}-\frac{1}{\delta}) \prod_{l=1}^{p} (c_l - c_{p+1}).
\end{align*}
The proof of Theorem \ref{th1.3} is completed by combining the eigenvalue expressions.
\end{proof}

\section{Conclusion}\label{sec.5}
This work establishes the dynamical stability of $n$-soliton solutions for the Intermediate Long Wave equation in $H^{n/2}(\mathbb{R})$ and the orbital stability of double solitons in $H^{1}(\mathbb{R})$. The proof constructs Lyapunov functionals $\mathcal{S}_n(u)$ from conservation laws and analyzes the spectral properties of the linearized operators $\mathcal{L}_n$. A key result is the characterization of their inertia indices: $\mathrm{in}(\mathcal{L}_n) = \bigl(\lfloor \frac{n+1}{2} \rfloor, n \bigr)$, obtained via recursion operators, which bridges the stability theories of the KdV ($\delta \to 0$) and BO ($\delta \to \infty$) limits.

Several challenges remain open. First, proving orbital stability for general $n$-solitons requires further work. Second, stability in critical spaces needs deeper investigation. Our spectral analysis relies on the implicit invertibility of recursion operators $\mathcal{R}(U^{(n)})$, which precludes explicit stability constants. The theory is also confined to the purely integrable case, without dissipation or periodic boundaries.

Future work should address the stability of periodic multi-solitons, study soliton interactions under quasi-periodic boundaries, and develop Floquet-Bloch methods for $\mathcal{L}_n$. Extending the recursion operator formalism to periodic settings could better connect ILW dynamics with the finite-gap theories of KdV and BO equations.

\vspace{1.5ex}
\noindent{\bf Acknowledgments}\\
\indent This work was supported by the National Natural Science Foundation of China under Grant No. 12371255, the Fundamental Research Funds for the Central Universities of CUMT under Grant No. 2024ZDPYJQ1003, and the Postgraduate Research \& Practice Program of Education \& Teaching Reform of CUMT under Grant No. 2025YJSJG031.

\vspace{1.5ex}
\noindent{\bf Data availability}\\
No data was used for the research described in the article.
\vspace{1.5ex}

\noindent{\bf Declaration of competing interest}\\
 The authors declare no conflict of interest.
\bibliographystyle{plain}

\end{document}